\documentclass[11pt]{amsart}

\usepackage{amsmath,amssymb,latexsym,soul,cite,mathrsfs}

\usepackage{color,enumitem,graphicx}
\usepackage[colorlinks=true,urlcolor=blue,
citecolor=red,linkcolor=blue,linktocpage,pdfpagelabels,
bookmarksnumbered,bookmarksopen]{hyperref}
\usepackage[english]{babel}
\usepackage{mathrsfs}
\usepackage[Conny]{fncychap}
\usepackage{color}
\usepackage{epsfig,longtable}
\usepackage{tikz}
\usetikzlibrary{arrows,calc}
\usepackage{relsize}
\usepackage{amssymb}

\usepackage[left=2.9cm,right=2.9cm,top=2.8cm,bottom=2.8cm]{geometry}
\usepackage[hyperpageref]{backref}

\usepackage[colorinlistoftodos]{todonotes}
\makeatletter
\providecommand\@dotsep{5}
\def\listtodoname{List of Todos}
\def\listoftodos{\@starttoc{tdo}\listtodoname}
\makeatother

\numberwithin{equation}{section}

\newtheorem{theorem}{Theorem}[section]
\newtheorem{proposition}[theorem]{Proposition}
\newtheorem{lemma}[theorem]{Lemma}

\newtheorem{remark}{Remark}
\newtheorem{definition}{Definition}[section]

\newtheorem*{theorem*}{Theorem}

\begin{document}

	\title[ASYMPTOTIC ANALYSIS FOR FRACTIONAL LAPLACIAN ON LONG CYLINDERS]
	{ASYMPTOTIC ISSUES FOR FRACTIONAL LAPLACIAN ON LONG
		CYLINDERS}
	\author{Tahir Boudjeriou and  Prosenjit Roy}
	\address[Tahir Boudjeriou]
	{\newline\indent
		Department of Basic Teaching
		\newline\indent Institute of Electrical \& Electronic Engineering
		\newline\indent University of boumerdes,  boumerdes, 35000, Algeria
		\newline\indent
		e-mail: {\tt t.boudjeriou@univ-boumerdes.dz}}
	\address[Prosenjit Roy]
	{\newline\indent
		Department of Mathematics and Statistics
		\newline\indent 	Indian Institute of Technology, Kanpur
		\newline\indent	UP, India, 2028016
		\newline\indent	e-mail: {\tt  prosenjit@iitk.ac.in}}

	\pretolerance10000
	
	
	\begin{abstract}
		\noindent
		In this paper, we are concerned with the asymptotic behavior of weak solutions to certain elliptic and parabolic problems involving the fractional $p$-Laplacian in cylindrical domains that become unbounded in one direction. The nonlocal nature of the operator describing the equations creates several technical difficulties in treating problems of this type.
		The main results, obtained within a nonlocal abstract framework, extend and complement related properties established in the local setting.
	\end{abstract}
	\thanks{}

	\maketitle	
	\section{Introduction and the main results}
	Recently, the study of the fractional Laplacian and related problems has received increasing attention. The interest in such problems is driven by their applications in continuum mechanics, minimal surfaces, finance, optimization, and game theory; see, for example, \cite{APP, Caf} and the references therein.
	
	In this paper, we are inspired by some established results concerning the asymptotic behavior of weak solutions as $\ell \rightarrow +\infty$ to the following problem
	\begin{equation}\label{eq22}\left\{
		\begin{array}{llc}
			(-\Delta)^{s}u_{\ell}=f & \text{in}\ & D_{\ell}, \\
			u_{\ell} =0 & \text{in} & \mathbb{R}^{N}\backslash D_{\ell},
		\end{array}\right.\tag{P}
	\end{equation}
	where $s\in (0,1)$, $D_{\ell} =(-\ell,\ell)^{m}\times \omega$, $1\leq m<N$, $\ell>0$, and $\omega$ is an open set in $\mathbb{R}^{N-m}$.
	
	To the best of our knowledge, the pioneering work on the asymptotic behavior of weak solutions to \eqref{eq22} as $\ell \rightarrow +\infty$ was carried out by Yeressian \cite{YA1}, who established the following result.
	\begin{theorem*}
		Let $u_{\ell}$ be the unique weak solution of (\ref{eq22}) for $s=\frac{1}{2}$, and assume the following conditions are satisfied :
		\begin{equation}
			\text{support}\,(f)\subset D_{\ell} \backslash D_{\ell-1}\quad \text{and}\quad \|f\|_{L^{2}(D_{\ell})}\leq 1.
		\end{equation}
		Then the following estimate holds
		$$\int_{D_{1}}u_{\ell}^{2}(x)\,dx\leq \frac{C}{\ell^{2}}\; \;\;\;\text{for all}\;\; \ell>0,$$
		where $C>0$ is a constant independent of  $\ell$.
	\end{theorem*}
	Problems of this type arise in scale analysis, which provides valuable insight into the behaviour of complex fluid systems in the regime where some of the characteristic dimensionless parameters become small or infinitely large.
	
	Later on, Chowdhury and Roy \cite{CH1} extended Yeressian's result to the case where $s \in (0,1)$. Furthermore, when $m = 1$ and the force term $f$ is assumed to be defined only on $\omega \subset \mathbb{R}^{N-1}$, that is, $f = f(x_{2}, x_{3}, \ldots, x_{N})$, the authors in \cite{CH1} also proved the following result.
	\begin{theorem*}
		Suppose that $s\in \left(\frac{1}{2}, 1\right)$ and $f(x_{2}, x_{3}, \ldots, x_{N})\in L^{2}(\omega)$. Let $u_{\ell}$ be the unique weak solution of (\ref{eq22}) for each $\ell$, and let $u_{\infty}$ be the unique weak solution to the following equation on the cross-section $\omega$ of the cylinder $D_{\ell}$
		
		\begin{equation}\label{eq2}\left\{
			\begin{array}{llc}
				(-\Delta')^{s}u_{\infty}=f (x_{2}, x_{3}, \ldots, x_{N}) & \text{in}\ & \omega, \\
				u_{\infty} =0 & \text{in} & \mathbb{R}^{N}\backslash \omega,
			\end{array}\right.
		\end{equation}
		where $(-\Delta')^{s}$ denotes the $N-1$ dimensional fractional Laplace operator. Then, for each $\alpha \in (0, 1)$, the following estimate holds
		$$ \int_{D_{\alpha \ell}}|u_{\ell}-u_{\infty}|^{2}\,dx\leq \frac{1}{\ell^{2s-1}}\;\; \;\;\text{for all }\;\; \ell >0.$$
	\end{theorem*}
	The above-mentioned results have been extended in \cite{TBB} to the case of parabolic equations involving the fractional Laplacian. In \cite{VA}, Ambrosio et al. considered the following elliptic problem
	\begin{equation}\label{eq223}\left\{
		\begin{array}{llc}
			(-\Delta_{\mathbb{R}^{N+k}})^{s}u_{\ell}=f_{\infty} & \text{in}\ & D^{N+k}_{\ell}, \\
			u_{\ell} =0 & \text{in} & \mathbb{R}^{N+k}\backslash D^{N+k}_{\ell},
		\end{array}\right.
	\end{equation}
	where $D^{N+k}_{\ell}=\omega^{N}\times B^{k}_{\ell}\subset \mathbb{R}^{N}\times \mathbb{R}^{k}$, $k,N\geq 1$, $\omega^{N}$ is a given bounded and Lipschitz domain in $\mathbb{R}^{N}$, and $B^{k}_{\ell}$ either the rectangle $(-\ell, \ell)^{k}$ or the Euclidean ball of radius $\ell$ centered at the origin. The operator $(-\Delta_{\mathbb{R}^{N+k}})^{s}$ denotes the $(N+k)$-dimensional fractional Laplace operator.  The authors in that paper obtained the following results:
	\begin{itemize}
		\item $ \lim\limits_{\ell\rightarrow+\infty}\inf\limits_{u\in H_{0}^{s}(D^{N+k}_{\ell})\backslash \{0\}}\frac{\left\langle  (-\Delta_{\mathbb{R}^{N+k}})^{s}u,u \right\rangle }{\|u\|_{L^{2}(D^{N+k}_{\ell})}^{2}}= \inf\limits_{u\in H_{0}^{s}(\omega^{N})\backslash \{0\}}\frac{\left\langle  (-\Delta_{\mathbb{R}^{N}})^{s}u,u \right\rangle }{\|u\|_{L^{2}(\omega^{N})}^{2}}=\lambda^{s}(\omega^{N}).$
		\item Let $f_{\infty}\in L^{2}(\omega^{N})$, and let $u_{\ell}$ be the unique weak solution of (\ref{eq223}). Then, 
		$$ \frac{1}{\ell^{k}B^{k}_{1}}\int_{B^{k}_{\ell}}u_{\ell}(x,t)\,dt\rightarrow u_{\infty}\;\;\text{strongly in}\;H_{0}^{s}(\omega)\; \;\text{as}\; \ell \rightarrow +\infty,$$
		where $u_{\infty}$ is the unique weak solution to the following problem:
		\begin{equation*}\label{eq2233}\left\{
			\begin{array}{llc}
				(-\Delta_{\mathbb{R}^{N}})^{s}u_{\infty}=f_{\infty} & \text{in}\ & \omega^{N}, \\
				u_{\infty} =0 & \text{in} & \mathbb{R}^{N}\backslash\omega^{N}.
			\end{array}\right.
		\end{equation*}
	\end{itemize}
	We emphasize that the asymptotic behavior of weak solutions to elliptic and parabolic equations involving the fractional Laplacian in domains becoming unbounded in one or several directions remains largely unaddressed in the literature and deserves further investigation. In fact, the analysis of such problems involving the fractional Laplacian is far from straightforward and significantly more involved than that of the corresponding local case, which was studied in the seminal paper by Chipot and Rougirel \cite{CH2}, where they considered the following parabolic problem
	\begin{equation}\label{eqqq1}\left\{
		\begin{array}{llc}
			\partial_{t}u_{\ell}-\text{div}(A(x,t)\nabla u_{\ell})=f(X_{2},t) & \text{in}\ & D_{\ell}\times (0,T), \\
			u_{\ell} =0 & \text{on} & \partial D_{\ell} \times (0,T), \\
			u_{\ell}(x,0)=u_{0}(X_{2})& \text{in} &D_{\ell},
		\end{array}\right.
	\end{equation}
	with $D_{\ell}:=(-\ell,\ell)^{p}\times \omega$, where $\omega$ is a bounded open subset of $\mathbb{R}^{N-p}$, $1\leq p< N$, $A(x,t)=(a_{ij})_{i,j=1, \ldots, N}$ is an $N\times N$ satisfying certain conditions. They proved that,  for any fixed $\ell_{0}>0$,  the unique weak solution $u_{\ell}$ of (\ref{eqqq1}) converges to $u_{\infty}$  in $L^{2}(0, T;L^{2}(D_{\ell_{0}}))$ and $L^{2}(0, T;H^{1}(D_{\ell_{0}}))$ with a speed faster than any power of $\frac{1}{\ell}$, where $u_{\infty}$ is the unique weak solution of the corresponding cross-section problem
	\begin{equation}\label{eqqq2}\left\{
		\begin{array}{llc}
			\partial_{t}u_{\infty}-\text{div}(A_{2,2}(x,t)\nabla_{X_{2}}u_{\infty})=f(X_{2},t) & \text{in}\ & \omega\times (0,T), \\
			u_{\infty} =0 & \text{on} & \partial \omega \times (0,T), \\
			u_{\infty}(x,0)=u_{0}(X_{2})& \text{in} &\omega.
		\end{array}\right.
	\end{equation}

Furthermore, the authors in \cite{CH2} investigated the asymptotic behavior of solutions as $\ell \rightarrow +\infty$ for a class of quasilinear parabolic equations. Subsequently, Chipot and Rougirel \cite{CH} addressed the same question for a class of elliptic equations. We also point out the works of Guesmia \cite{G1, G2}, Chipot and Xie \cite{M1}, Esposito et al. \cite{ESP}, Jana \cite{VPP}, Rawat et al. \cite{VPPR}, as well as those of Alves, Figueiredo, and Furtado \cite{alves}, and Figueiredo, Pimenta, and Siciliano \cite{figue}, which deal with multiplicity results for stationary equations in expanding domains.
	
	To the best of our knowledge, there is currently no work investigating the asymptotic behavior of weak solutions as $\ell \rightarrow +\infty$ for elliptic and parabolic equations involving the fractional $p$-Laplacian in domains that become unbounded. Motivated by this gap and the aforementioned studies, the main objective of this paper is to analyze the asymptotic behavior of weak solutions to the following nonlocal elliptic problem as $\ell \rightarrow +\infty$:
	\begin{equation}\label{EQ1}\left\{
		\begin{array}{llc}
			(-\Delta)_{p}^{s}u_{\ell}(x)=f(x) & \text{in}\ & \Omega_{\ell}, \\
			u_{\ell} (x)=0 & \text{in} & \mathbb{R}^{N}\backslash\Omega_{\ell},
		\end{array}\right.
	\end{equation}
	as well as the corresponding nonlocal parabolic problem
	\begin{equation}\label{EQ2}\left\{
		\begin{array}{llc}
			\partial_{t}u_{\ell}(x,t)+(-\Delta)_{p}^{s}u_{\ell}(x,t)=f(x,t) & \text{in}\ & \Omega_{\ell}\times (0,T), \\
			u_{\ell}(x,t) =0 & \text{in} & \left(\mathbb{R}^{N}\backslash\Omega_{\ell}\right)\times (0,T), \\
			u_{\ell}(x,0)=u_{0}(x)& \text{in} &\Omega_{\ell},
		\end{array}\right.
	\end{equation}
	where $T>0$,  $\partial_{t}=\partial / \partial t$, and the leading operator $(-\Delta)_{p}^{s}\psi$ is the fractional $p$-Laplace operator defined for smooth functions by
	\begin{equation}\label{FR}
		(-\Delta)_{p}^{s}\psi(x) := C_{N,s,p}\lim_{\epsilon \rightarrow 0^{+}} \int_{\mathbb{R}^{N}\backslash B_{\epsilon}(x)}\frac{|\psi(x)-\psi(y)|^{p-2}(\psi(x)-\psi(y))}{|x-y|^{N+sp}}\,dy\;\;\;\; x\in \mathbb{R}^{N},
	\end{equation}
	where $s\in (0, 1)$, and $B_{\epsilon} (x)$ denotes the open ball in $\mathbb{R}^{N}$ centred at $x\in \mathbb{R}^{N} $ with radius $\epsilon > 0$. The constant $C_{N,s,p}$ (see, e.g., \cite{ANT}) is given by
	\begin{equation}\label{CN}
		C_{N,s,p}=\frac{sp2^{2s-1}\Gamma \left(\frac{N+ps}{2}\right)}{2\pi^{\frac{N-1}{2}}\Gamma (1-s)\Gamma (\frac{p+1}{2})},	
	\end{equation}
	where $\Gamma$ denotes the usual Gamma function. 
	
	In the sequel, we introduce some notations that will be used throughout the paper. For
	$x=(x_{1}, x_{2}, \ldots, x_{N})\in \mathbb{R}^{N}$, we set
	\begin{equation}
		x=(x_{1}, X_{2}), \;\;X_{2}=(x_{2}, x_{3}, \ldots, x_{N}).
	\end{equation}
	Let $\ell>0$. We denote by $ \Omega_{\ell}=(-\ell, \ell)\times \omega \subset\mathbb{R}^{N}$ ($N\geq 2$) the cylinder of length $\ell$, where  $\omega \subset  \mathbb{R}^{N-1}$ is a bounded open set representing the cross-section. A schematic diagram of the domain
	$\Omega_{\ell}$ is shown in \textbf{Fig.1.}
	\begin{center}
		\begin{tikzpicture}[
			scale=1.7,
			IS/.style={blue, thick},
			LM/.style={red, thick},
			axis/.style={very thick, ->, >=stealth', line join=miter},
			important line/.style={thick}, dashed line/.style={dashed, thin},
			every node/.style={color=black},
			dot/.style={circle,fill=black,minimum size=4pt,inner sep=0pt,
				outer sep=-1pt},
			]
			\coordinate (tint) at (2.5,0);
			\coordinate (beg_2) at (-1.8,0);
			\coordinate (beg_3) at (0.15,0);
			\coordinate (M) at (-1.1,0.45);
			\coordinate (M2) at (1.08,0.45);
			\coordinate (M3) at (-2.68,2.3);
			\coordinate (M4) at (2,2.3);
			\coordinate (t) at (1.85,0);
			\coordinate (tt) at (-1.5,0);
			\coordinate (M5) at (0,1.89);
			\draw[black, thick] (-1.3,0.8)--(0,0.8)--(1.75,0.8);
			\draw[black, thick] (-1.3,-0.8)--(0,-0.8)--(1.75,-0.8);
			\draw[thick,->] (-2,0) -- (2.5,0) node[anchor=north west] {$x_{1}$};
			\draw[thick,->] (0.27,0) -- (0.27,1.5) node[anchor=south east] {$X_{2}$};

			\draw[black,thick](-1.3,0)--(-1.3,0.8);
			\draw[black, thick](1.75,0)--(1.75,0.8);
			\draw[black, thick](0.27,-1)--(0.27,0);
			\draw[black, thick](0.27,-1)--(0.27,0);
			\draw[black, thick](0.27,-1)--(0.27,0);
			\draw[black, thick](0.27,-1)--(0.27,0);
			\draw[black, thick](0.27,-1)--(0.27,0);
			\draw[black, thick](0.27,-1)--(0.27,0);
			\draw[black,  thick](-1.3, -0.8)--(-1.3,0);
			\draw[black, thick](1.75, -0.8)--(1.75,0);

			\fill[black](beg_3)node[below]{$0$};

			\fill[black](t)  node[below]{$\ell$};
			\fill[black](tt)  node[below]{$-\ell$};
			
		\end{tikzpicture}
	\end{center}
	\begin{center}
		\textbf{Fig.1.}
	\end{center}
	Next, we recall some known facts about fractional Sobolev spaces. Let $\Omega$ be an open set in $\mathbb{R}^{N}$. For any $s\in(0,1)$ and $p\geq 2$, the space $W^{s,p}(\Omega)$ is defined by
	$$ W^{s,p}(\Omega)=\left\{u\in  L^{p}(\Omega):\, [u]_{s,p, \mathbb{R}^{N}}< +\infty \right\}, $$
	where $[u]_{s,p, \mathbb{R}^{N}}$ is the so-called Gagliardo seminorm, given by
	$$[u]^{p}_{s,p, \mathbb{R}^{N}}:=\int_{\mathbb{R}^{N}\times \mathbb{R}^{N}} \frac{|u(x)-u(y)|^{p}}{|x-y|^{N+sp}}\,dxdy.$$
	We define the space  $W_{0}^{s,p}(\Omega)$ by
	$$ W_{0}^{s,p}(\Omega)=\left\{u\in W^{s,p}(\Omega):\ u=0\;\;\text{a.e. in}\;\mathbb{R}^{N}\backslash \Omega \right\}$$
	equipped with the norm
	$$\|u\|:=\left(\|u\|^{p}_{L^{p}(\Omega)}+[u]^{p}_{s,p, \mathbb{R}^{N}}\right)^{\frac{1}{p}}.$$
	Functions in $W_{0}^{s,p}(\Omega)$ can be regarded as elements of $W_{0}^{s,p}(\mathbb{R}^{N})$ by extending them by zero outside $\Omega$. For $1\leq p<\frac{N}{s}$, the fractional critical exponent is defined by $p^{*}_{s}=\frac{Np}{N-sp}$. If $1\leq r\leq p^{*}_{s}$, we have the fractional Sobolev embedding $W_{0}^{s,p}(\Omega)\hookrightarrow L^{r}(\Omega)$, which is compact when $1\leq r<p^{*}_{s}$.
	For more details on fractional Sobolev spaces and nonlocal fractional equations, we refer to the monograph by Molica Bisci, Rădulescu, and Servadei \cite{V1}.
	
	For each $T>0$, we consider the parabolic space
	$$ W(0, T;W_{0}^{s,p}(\Omega))=\left\{u \in L^{\infty}(0, T; W_{0}^{s,p}(\Omega)), \; u_{t} \; \text{exists and}\; u_{t}\in L^{2}(0, T; L^{2}(\Omega)\right\}. $$
	
	In the rest of the paper, we shall use the following notations:
	\begin{itemize}
		\item  For every $p>1$, $\|\,.\,\|_{p,\Omega}$ denotes the norm in $L^{p}(\Omega)$.
		\item $u(t)=u(.,t)$ for any $t\in [0,T]$.
		\item If $X$ and $Y$ are two quantities (typically non-negative ), we write $X\precsim Y$ or $Y\succsim X$ to mean  the   that $X\leq C Y$  for a constant $C>0$ which may vary from line to line but is independent of $\ell$.
		\item Let $1<p<\infty$, we denote by $p'=\frac{p}{p-1}$  the conjugate of the exponent of $p$.
		\item For $ p>1$, we define the monotone function $\varphi_{p} :\mathbb{R} \rightarrow\mathbb{R} $ by $\varphi_{p}(r)=|r|^{p-2}r$ for any $r\in \mathbb{R}.$
		\item In some cases, we write $d\mu(x,y)=\frac{dxdy}{|x-y|^{N+sp}}.$
	\end{itemize}
	\subsection{The elliptic problem (\ref{EQ1})} We now present our first main result, starting with the definition of weak solutions to the elliptic problem (\ref{EQ1}).
	\begin{definition}\label{DEF1}
		Let $f\in L^{2}(\Omega_{\ell})$. We say that $u_{\ell}\in  W_{0}^{s,p}(\Omega_{\ell})$ is a weak solution of (\ref{EQ1}) if, for every $v\in W_{0}^{s,p}(\Omega_{\ell})$,  the following identity holds
		\begin{equation}\label{IDNT1}
			\frac{C_{N,s,p}}{2}\int_{\mathbb{R}^{N}\times \mathbb{R}^{N}}\frac{\varphi_{p}(u_{\ell}(x)-u_{\ell}(y))(v(x)-v(y))}{|x-y|^{N+sp}}\,dxdy= \int_{\Omega_{\ell}}fv\,dx.
		\end{equation}
	\end{definition}
	The existence and uniqueness of a weak solution to (\ref{EQ1}) can be established by following the same approach as in \cite{DH2}, applying the direct method in the calculus of variations to the functional $E : W_{0}^{s, p}(\Omega)\rightarrow \mathbb{R}$ given by
	$$ E(v)=\frac{C_{N,s,p}}{2p}\int_{\mathbb{R}^{N}\times \mathbb{R}^{N}} \frac{|v(x)-v(y)|^{p}}{|x-y|^{N+sp}}\,dxdy-\int_{\Omega_{\ell}} fv\,dx.$$
	
	Assume that the function $f$  depends only on  $X_{2}$.  We aim to study the asymptotic behavior of $u_{\ell} $ as $\ell \rightarrow+\infty$ on a fixed domain $\Omega_{\ell_{0}}$. More precisely, we intend to show that $u_{\ell} $ converges to $u_{\infty} $ in the $L^{p}$-norm where $u_{\infty} $  is the unique weak solution of
	\begin{equation}\label{EQQ1}\left\{
		\begin{array}{llc}
			(-\Delta')_{p}^{s}u_{\infty}(X_{2})=f(X_{2}) & \text{in}\ & \omega, \\
			u_{\infty}(X_{2}) =0 & \text{in} & \mathbb{R}^{N-1}\backslash\omega,
		\end{array}\right.
	\end{equation}
	where  $(-\Delta')^{s}_{p}$ denotes the $N-1$ dimensional fractional $p$-Laplace operator.
	
	For $f\in L^{2} (\omega)$, we say that $u_{\infty}\in W_{0}^{s,p}(\omega)$ is a weak solution of (\ref{EQQ1}) if, for any $v\in W_{0}^{s, p}(\omega)$ there holds
	\begin{equation}\label{IDNTT1}
		\frac{C_{N-1,s,p}}{2}\int_{\mathbb{R}^{N-1}\times \mathbb{R}^{N-1}}\frac{\varphi_{p}(u_{\infty}(X_{2})-u_{\infty}(Y_{2}))(v(X_{2})-v(Y_{2}))}{|X_{2}-Y_{2}|^{N-1+sp}}\,dX_{2}dY_{2}= \int_{\omega}fv\,dX_{2}.
	\end{equation}
	We note that the existence and uniqueness of a weak solution to (\ref{EQQ1}) can be established in the same manner as above, by applying the direct method in the calculus of variations.
	
	The first main result is stated as follows.
	
	\begin{theorem}\label{THE}
		Let $u_{\ell}$ be the unique weak solution of (\ref{EQ1}) and $u_{\infty}$ the unique weak solution of (\ref{EQQ1}). 	
		Assume that  $f\in L^{2}(\omega)$  is independent of $x_{1} $, that is,  $ f(x)=f(X_{2})$.
		Moreover, if $p>2$ and $s\in \left(\frac{1}{p'}, 1\right)$, then for any fixed $\ell_{0}>0$ there holds
		\begin{equation*}\label{FIN}
			\|u_{\ell}-u_{\infty}\|_{p, \Omega_{\ell_{0}}}\precsim \left[\frac{1}{\ell^{s(p-1)-\frac{p-1}{p}}}+\frac{1}{\ell^{s-\frac{p-1}{p}}}\right]^{\frac{1}{p-1}},
		\end{equation*}
		for all sufficiently large $\ell >0$.
	\end{theorem}
	\begin{remark}
		It is not known whether the result stated in Theorem \ref{THE} can be extended to more general domains of the form $\Omega_{\ell}=(-\ell,\ell)^{m}\times\omega \subset \mathbb{R}^{m}\times \mathbb{R}^{N-m}$ for any $m\geq 2$. Furthermore, its validity for $s\in (0,1)$ also remains an open question.
	\end{remark}
	We recall that, for the classical  $p$-Laplace operator with Dirichlet boundary conditions, Chipot and  Xie \cite[Theorem 2.1]{M1} proved the  estimate
	$$ \|\nabla \left(u_{\ell}- u_{\infty}\right)\|_{p,\Omega_{\ell_{0}}}\precsim \frac{\|\partial_{x_{2}}u_{\infty}\|_{p,\omega}}{\ell^{\frac{2}{p(p-2)}}},$$
	where $p>2$ and $u_{\infty}$ denotes the unique solution of
	\begin{equation*}\label{EQ1bis}\left\{
		\begin{array}{llc}
			-\Delta_{p}u_{\infty}(x_{2})=f(x_{2}) & \text{in}\ & (-1,1), \\
			u_{\infty}\in W_{0}^{1,p}(-1,1).
		\end{array}\right.
	\end{equation*}
	\subsection{The parabolic problem (\ref{EQ2})} In this subsection, we present our second main result, which concerns the asymptotic behavior of weak solutions to the parabolic problem  (\ref{EQ2}).
	\begin{definition}\label{DEF}
		Let $u_{0}\in L^{2}(\Omega_{\ell})$ and $f\in L^{2}(0,T; L^{2}(\Omega_{\ell}))$. We say that $u_{\ell}\in W(0, T; W_{0}^{s,p}(\Omega_{\ell}))$ is a weak solution of (\ref{EQ2}) if,
		for every $v\in W_{0}^{s,p}(\Omega_{\ell})$ the following identity holds for a.e. $ t\in (0, T)$
		\begin{equation}\label{IDNT2}
			\int_{\Omega_{\ell}}\partial_{t}u_{\ell}(t) v\,dx +\frac{C_{N,s,p}}{2}\int_{\mathbb{R}^{N}\times \mathbb{R}^{N}}\frac{\varphi_{p}(u_{\ell}(x,t)-u_{\ell}(y,t))(v(x)-v(y))}{|x-y|^{N+sp}}\,dxdy= \int_{\Omega_{\ell}}f(t)v\,dx.
		\end{equation}
		Moreover, the initial condition
		
		\begin{equation}
			u_{\ell}(0)=u_{0},
		\end{equation}
		is satisfied.
	\end{definition}
	If the source term $f$ depends only on $X_{2}$ and $t$, and the initial datum depends only on $X_{2}$, we aim to show  that the unique weak solution of (\ref{EQ2}) converges, in the $L^{p}$-norm, to the unique weak solution of the following problem:
	\begin{equation}\label{eq3}\left\{
		\begin{array}{llc}
			\partial_{t}u_{\infty}(X_{2},t)+(-\Delta')_{p}^{s}u_{\infty}(X_{2},t)=f (X_{2},t) & \text{in}\ & \omega \times (0,T), \\
			u_{\infty}(X_{2},t) =0 & \text{in} & \left(\mathbb{R}^{N-1}\backslash \omega\right)\times (0, T), \\
			u_{\infty}(X_{2},0)=u_{0}(X_{2}) & \text{in} & \omega.
		\end{array}\right.
	\end{equation}
	The weak solution of problem (\ref{eq3}) is defined similarly to Definition \ref{DEF}, but for the reader’s convenience, we present the full statement here.
	\begin{definition}\label{Def1}
		Let $u_{0}\in L^{2}(\omega)$ and $f\in L^{2}(0,T; L^{2}(\omega))$. We say that $u_{\infty}\in W(0, T; H_{0}^{s}(\omega))$ is a weak solution of (\ref{eq3}) if, 
		for all $v\in W_{0}^{s,p}(\omega)$, the following holds for a.e. $ t\in (0, T)$
		\begin{equation}\label{IDN1}
			\begin{array}{ll} \displaystyle		\int_{\omega}\partial_{t} u_{\infty}(t)v\,dX_{2}+ &\displaystyle \frac{C_{N-1,s,p}}{2}\int_{\mathbb{R}^{N-1}\times \mathbb{R}^{N-1}}\frac{\varphi_{p}(u_{\infty}(X_{2}, t)-u_{\infty}(Y_{2}, t))(v(X_{2})-v(Y_{2}))}{|x-y|^{N-1+sp}}\,dX_{2}dY_{2}=\\
				&\displaystyle  \int_{\omega}f(X_{2},t)v\,dX_{2}.\end{array}
		\end{equation}
		Moreover,
		\begin{equation}
			u_{\infty}(0)=u_{0}(X_{2}).
		\end{equation}	
	\end{definition}
	To solve the parabolic problem (\ref{EQ2}), we can reformulate it as a first-order abstract Cauchy problem in $H=L^{2}(\Omega_{\ell})$ by introducing the functional $\psi$ from $H$ to $(-\infty, +\infty]$ defined by
	$$
	\psi(u_{\ell})=\left\{
	\begin{array}{l}
		\displaystyle	\frac{C_{N,s,p}}{2p}\int_{\mathbb{R}^{N}\times \mathbb{R}^{N}}\frac{|u_{\ell}(x,t)-u_{\ell}(y,t)|^{p}}{|x-y|^{N+sp}}\,dxdy,\quad \mbox{if} \quad u_{\ell} \in W_{0}^{s,p}(\Omega_{\ell}),\\
		\mbox{}\\
		+\infty, \quad \mbox{if} \quad  u_{\ell} \in H \setminus W_{0}^{s,p}(\Omega_{\ell}).
	\end{array}
	\right.
	$$
	It is straightforward to verify that 
	$\psi$ is convex, lower semi-continuous, and proper.
	It was shown in \cite[Theorem 2.3]{DH2} that solving (\ref{EQ2}) in the sense of Definition \ref{DEF} is equivalent to solve the following abstract Cauchy problem:
	\begin{equation}\label{EQC}\left\{
		\begin{array}{llc}
			\displaystyle	\frac{du_{\ell}(t)}{dt}+\partial \psi (u_{\ell}(t))\ni f(t)& \text{in}\ & H, \; 0<t<T, \\
			u_{\ell}(x,0)=u_{0},&  &
		\end{array}\right.
	\end{equation}
	where $\partial \psi (u_{\ell})$ denotes the subdifferential of $\psi$ at  $u_{\ell}$ in the sense of convex analysis. Since $u_{0}\in L^{2}(\Omega_{\ell})$, it follows from \cite[Theorem 3.6 and Lemma 3.3]{H} that there exists $T>0$ and a unique solution $u_{\ell}$ to the Cauchy problem (\ref{EQC}).
	
	Similarly, the existence and uniqueness of a weak solution to (\ref{eq3}) can be obtained by considering the functional $\phi$ from $\mathcal{H}=L^{2}(\omega)$ to $(-\infty, +\infty]$ as
	$$
	\phi(u_{\infty})=\left\{
	\begin{array}{l}
		\displaystyle	\frac{C_{N-1,s,p}}{2p}\int_{\mathbb{R}^{N-1}\times \mathbb{R}^{N-1}}\frac{|u_{\infty}(X_{2},t)-u_{\infty}(Y_{2},t)|^{p}}{|X_{2}-Y_{2}|^{N-1+sp}}\,dX_{2}dY_{2},\quad \mbox{if} \quad u_{\infty} \in W_{0}^{s,p}(\omega),\\
		\mbox{}\\
		+\infty, \quad \mbox{if} \quad  u_{\infty}\in \mathcal{H} \setminus W_{0}^{s,p}(\omega).
	\end{array}
	\right.
	$$
	The second main result of this paper is presented in the following theorem.
	\begin{theorem}\label{THP}
		Let $u_{\ell}$ and  $u_{\infty}$  be the unique weak solutions of (\ref{EQ2}) and (\ref{eq3}), respectively.	Assume that the functions $f\in L^{2}(0,T;L^{2}(\omega))$ and $u_{0}\in L^{2}(\omega)$ are independent of $x_{1} $, that is,
		$$ f(x,t)=f(X_{2}, t), \;\;\;\;u_{0}(x)=u_{0}(X_{2}).$$
		If $p>2$ and $s\in \left(\frac{1}{p'}, 1\right)$, then for any fixed $\ell_{0}>0$ we have
		\begin{equation*}
			\|u_{\ell}-u_{\infty}\|^{2}_{L^{\infty}(0, T;L^{2}(\Omega_{\ell_{0}}))}+ \|u_{\ell}-u_{\infty}\|^{p}_{L^{p}(0, T;L^{p}(\Omega_{\ell_{0}}))}\precsim\frac{1}{\ell^{sp-1}}+\frac{1}{\ell^{\frac{sp}{p-1}-1}},
		\end{equation*}
		for all sufficiently large 
		$\ell >0$.
	\end{theorem}
	The proofs of Theorems \ref{THE} and \ref{THP} rely on energy estimates combined with a fractional $p$-Poincaré inequality (see Lemma \ref{LLe}) and an appropriate choice of a test function in the weak formulation.
	
	As noted by Yeressian \cite{YA1}, the estimates established in Theorems \ref{THE} and \ref{THP} play an important role in numerical computations of solutions in large domains, particularly when the interest lies in computing the solution within a smaller subdomain. Furthermore, these estimates are also relevant for proving the well-posedness of equations in unbounded domains with right-hand sides that do not decay at infinity.
	
	The remainder of the paper is organized as follows. In Section $2$, we introduce some preliminaries that will be essential for the proofs of the main results. The subsequent sections are devoted to the proofs of these results.

	\section{Preliminary Results}
	In this section, we present some preliminary results that will be used in the sequel.
	
	\begin{proposition}(\cite[Proposition 1.2]{RE})\label{Pr1}
		Let $V$ be a Banach space that is dense and continuously embedded in the Hilbert space $H$. We identify $H=H^{'}$ so that $V\hookrightarrow H=H^{'}\hookrightarrow V^{'}$. Then the Banach space $$W_{p}:=\{u\in L^{p}(0,T; V), \; u_{t}\in L^{p'}(0,T, V')\}$$ is contained in $C([0,T]; H)$. Moreover,  if $u\in W_{p}$,  $\|u(t)\|_{L^{2}(\Omega)}$ is absolutely continuous on $[0,T]$, and
		$$ \frac{d}{dt}\|u(t)\|^{2}_{L^{2}(\Omega)}=2\langle u_{t}(t), u(t) \rangle, \; \text{a.e. on}\; [0,T]. $$
		In addition, there exists a constant $C > 0$ such that
		$$ \|u\|_{C(0,T; H)}\leq C\|u\|_{W_{p}}, \;\;\text{for all }\, u\in W_{p}.$$
	\end{proposition}
	Next, we recall the fractional $p$-Poincaré inequality due to Mohanta and Sk \cite[Theorem 1.2]{KK2}, where the best constant is also established.
	
	\begin{lemma}[Poincar\'e Inequality]\label{LLe} Consider the strip $D_{\infty}=\mathbb{R}^{m}\times \omega\subset \mathbb{R}^{N}$  with $1\leq m<N$, where $\omega$ is a bounded open subset of $\mathbb{R}^{N-m}$. Then, for $0<s<1$ and $1<p<\infty$, we have
		\begin{equation}\label{PC}
			P_{N,s,p}^{2}(D_{\infty}):=\inf_{u\in W_{0}^{s,p}(D_{\infty})\backslash \{0\}}\frac{[u]^{p}_{s,p, \mathbb{R}^{N}}}{\|u\|^{p}_{p,D_{\infty}}}=P_{N-m,s,p}^{2}(\omega):=\inf_{u\in W_{0}^{s,p}(\omega)\backslash \{0\}}\frac{[u]^{p}_{s,p,\mathbb{R}^{N-m}}}{\|u\|^{p}_{p,\omega}}>0.
		\end{equation}
	\end{lemma}
	This result generalizes the fractional Poincaré inequality obtained in \cite{II2} for the spaces $H_{0}^{s}(\Omega)$, where the optimal constant is obtained. In \cite{KK}, the authors further extended this inequality to the setting of fractional Orlicz–Sobolev spaces.
	
	\begin{lemma}\label{Lem}
		For each $N\in \mathbb{N}$ and $s\in (0,1)$, let $C_{N,s,p}$ be the constant defined in (\ref{CN}). Then one has $C_{N,s,p}\theta_{N,p}=C_{N-1,s,p}$, where
		\begin{equation}\label{EZ}
			\theta_{N,p}=\int_{\mathbb{R}}\frac{dz}{(1+z^{2})^{\frac{N+sp}{2}}}.
		\end{equation}
		Moreover, if $a>0$ then
		$$ \int_{\mathbb{R}} \frac{dz}{\left(1+\frac{|x-z|^{2}}{a^{2}}\right)^{\frac{N+sp}{2}}}=a\theta_{N,p}$$
	\end{lemma}
	\begin{proof}
		The proof is similar to that in \cite{KK2}; however, we include it here for the sake of completeness. Indeed, from \cite[Theorem 8.20]{RUD} we have
		\begin{equation}\label{EqQ}
			\int_{0}^{\frac{\pi}{2}}(\sin \eta)^{2p-1} (\cos \eta)^{2q-1}\,d\eta=\frac{\Gamma (p)\Gamma (q)}{2\Gamma (p+q)} \;\;\text{for all }\, p,q>0,
		\end{equation}
		where $\Gamma$ denotes the gamma function. Further, by using the change of variables $z=\tan(\eta) $ in (\ref{EZ}), we obtain
		$$ \theta_{N,p}=2\int_{0}^{+\infty}\frac{dz}{(1+z^{2})^{\frac{N+sp}{2}}}=2\int_{0}^{\frac{\pi}{2}}(\cos \eta)^{N+sp-2}\,d\eta=\frac{\sqrt{\pi}\Gamma \left(\frac{N+sp-1}{2}\right)}{\Gamma \left(\frac{N+sp}{2}\right)}.$$
		Hence, by the definition of $C_{N,s,p}$ in (\ref{CN}), we deduce the desired result.  The second statement follows by making the change of variables $y=\frac{x-z}{a}.$
	\end{proof}
	In the next lemma, we recall some well-known elementary inequalities that will be used in the following section.
	\begin{lemma}\label{ILM} For any $a,b\in \mathbb{R}_{+}$ and $c,d\in \mathbb{R}$, we have
		\begin{enumerate}
			\item $ \big(|c|^{q-2}c-|d|^{q-2}d\big)(c-d)\geq |c-d|^{q}, \;\;\;\forall q\geq 2,$	
			\item 	$ (a+b)^{q}\leq 2^{q-1}(a^{q}+b^{q}),\;\; \;\forall q\geq 1,$
			\item $ (a+b)^{q}\leq a^{q}+b^{q}, \;\;\;  0\leq q\leq 1,$
			\item $| a^{q} -b^{q} |\leq q|a-b|\left[a^{q-1}+b^{q-1}\right],\;\;\;\forall q\geq 1.$
			\item $ ||c|^{q-2}c -|d|^{q-2}d|\leq C(q)|c-d|\left[|c|^{q-2}+|d|^{q-2}\right], \;\;\forall q>1.$
		\end{enumerate}
	\end{lemma}
	\section{Proof of Theorem \ref{THE}}
	This section is devoted to proving Theorem \ref{THE}. We begin with some preparatory lemmas.
	
	\begin{lemma}\label{LEm}
		Let $u_{\ell}$ be the weak solution to (\ref{EQ1bis}). Then, there holds
		\begin{equation}
			\|u_{\ell}\|^{p}_{p,\Omega_{\ell}}\precsim \int_{\mathbb{R}^{N}\times \mathbb{R}^{N}} \frac{|u_{\ell}(x)-u_{\ell}(y)|^{p}}{|x-y|^{N+sp}}\,dxdy\precsim  \ell \|f\|_{\omega,2}^{\frac{p}{p-1}}.
		\end{equation}
	\end{lemma}
	\begin{proof}
		The first inequality follows immediately from the fractional  $p$-Poincar\'e inequality. To prove the second part, we take $v=u_{\ell}$ in (\ref{IDNT1}) and apply Hölder's inequality, obtaining
		$$ \frac{C_{N,s,p}}{2} \int_{\mathbb{R}^{N}\times \mathbb{R}^{N}} \frac{|u_{\ell}(x)-u_{\ell}(y)|^{p}}{|x-y|^{N+ps}}\,dxdy\leq \|u_{\ell}\|_{p,\Omega_{\ell}}\left(\int_{-\ell}^{\ell}\int_{\omega}|f(X_{2})|^{p'}\,dx\right)^{\frac{1}{p'}},$$
		where $p'=\frac{p}{p-1}$.
		By the fractional $p$-Poincar\'e inequality we derive
		\begin{align*}
			&\frac{C_{N,s,p}}{2}\int_{\mathbb{R}^{N}\times \mathbb{R}^{N}}  \frac{|u_{\ell}(x)-u_{\ell}(y)|^{p}}{|x-y|^{N+ps}}\,dxdy\\
			&\leq C_{P} \left(\int_{\mathbb{R}^{N}\times \mathbb{R}^{N}}  \frac{|u_{\ell}(x)-u_{\ell}(y)|^{p}}{|x-y|^{N+ps}}\,dxdy\right)^{\frac{1}{p}}\left(\int_{-\ell}^{\ell}\int_{\omega}|f(X_{2})|^{p'}\,dX_{2}dx_{1}\right)^{\frac{1}{p'}},
		\end{align*}
		
		where $C_{P}>0$  is the best constant in the fractional $p$-Poincar\'e inequality. It follows that
		$$
		\frac{C_{N,s,p}}{2}\left(\int_{\mathbb{R}^{N}\times \mathbb{R}^{N}}  \frac{|u_{\ell}(x)-u_{\ell}(y)|^{p}}{|x-y|^{N+ps}}\,dxdy\right)^{\frac{1}{p'}}\leq 2^{\frac{1}{p'}}\ell^{\frac{1}{p'}}C_{P} \left(\int_{\omega}|f(X_{2})|^{p'}\,dX_{2}\right)^{\frac{1}{p'}}.$$
		Since $p'<2$,  Hölder's inequality yields
		$$
		\frac{C_{N,s,p}}{2}\left(\int_{\mathbb{R}^{N}\times \mathbb{R}^{N}}  \frac{|u_{\ell}(x)-u_{\ell}(y)|^{p}}{|x-y|^{N+ps}}\,dxdy\right)^{\frac{1}{p'}}\leq 2^{\frac{1}{p'}}\ell^{\frac{1}{p'}}C_{P}|\omega|^{\frac{p-2}{2(p-1)}} \|f\|_{2,\omega}.$$
		Therefore,
		\begin{equation}
			\int_{\mathbb{R}^{N}\times \mathbb{R}^{N}}  \frac{|u_{\ell}(x)-u_{\ell}(y)|^{p}}{|x-y|^{N+sp}}\,dxdy\leq \ell K \|f\|_{\omega,2}^{\frac{p}{p-1}},
		\end{equation}
		where $K=2^{p'+1}C^{p'}_{P}C^{-p'}_{N,s,P}|\omega|^{\frac{(p-2)p}{2(p-1)^{2}}}.$ Thus, This completes the proof.
	\end{proof}
	The following lemma plays a key role in the proof of Theorem \ref{THE}.
	\begin{lemma}\label{LE}
		Let $x=(x_{1}, X_{2})\in \Omega_{\ell}$ and $u_{\infty}$  denotes the unique weak solution of (\ref{EQQ1}) . Then, for any $v\in W_{0}^{s,p}(\Omega_{\ell})$, there holds
		\begin{equation}\label{DFFR}
			\frac{C_{N,s}}{2}	\int_{\mathbb{R}^{N}\times \mathbb{R}^{N} }\frac{\varphi_{p}\left(u_{\infty}(X_{2})-u_{\infty}(Y_{2})\right)(v(x)-v(y))}{|x-y|^{N+sp}}\,dxdy=\int_{\Omega_{\ell}}fv\,dx.
		\end{equation}
		
	\end{lemma}
	\begin{proof} By Lemma \ref{Lem} and (\ref{IDNTT1}), together with straightforward calculations, we have
		\begin{align*}
			&\int_{\Omega_{\ell}} f(X_{2})v(x)\,dx
			=\int_{\mathbb{R}}\int_{\omega} f(X_{2})v(x_{1}, X_{2})\,dX_{2}dx_{1}\\
			=&\frac{C_{N,s,p}\theta_{N,p}}{C_{N-1,s,p}}C_{N-1,s,p}\int_{\mathbb{R}}\int_{\mathbb{R}^{N-1}} \int_{\mathbb{R}^{N-1}}\frac{|u_{\infty}(X_{2})-u_{\infty}(Y_{2})|^{p-2}(u_{\infty}(X_{2})-u_{\infty}(Y_{2}))v(x_{1}, X_{2})}{|X_{2}-Y_{2}|^{N-1+sp}}\,dY_{2}dX_{2}dx_{1}\\
			=& C_{N,s,p}\theta_{N,p}\int_{\mathbb{R}}\int_{\mathbb{R}^{N-1}} \int_{\mathbb{R}^{N-1}}\frac{|u_{\infty}(X_{2})-u_{\infty}(Y_{2})|^{p-2}(u_{\infty}(X_{2})-u_{\infty}(Y_{2}))v(x_{1}, X_{2})}{|X_{2}-Y_{2}|^{N-1+sp}}\,dY_{2}dX_{2}dx_{1}\\
			=&C_{N,s,p}\int_{\mathbb{R}^{N}} \int_{\mathbb{R}^{N-1}} \frac{|u_{\infty}(X_{2})-u_{\infty}(Y_{2})|^{p-2}(u_{\infty}(X_{2})-u_{\infty}(Y_{2}))v(x)}{|X_{2}-Y_{2}|^{N+sp}}\left(\int_{\mathbb{R}}\frac{dy_{1}}{\left(1+\frac{|x_{1}-y_{1}|^{2}}{|X_{2}-Y_{2}|^{2}}\right)^{\frac{N+sp}{2}}}\right)\,dY_{2}dx \\
			=&C_{N,s,p}\int_{\mathbb{R}^{N}} \int_{\mathbb{R}^{N}} \frac{|u_{\infty}(X_{2})-u_{\infty}(Y_{2})|^{p-2}(u_{\infty}(X_{2})-u_{\infty}(Y_{2}))v(x)}{|x-y|^{N+sp}}\,dydx\\=&\frac{C_{N,s,p}}{2}\int_{\mathbb{R}^{N}} \int_{\mathbb{R}^{N}} \frac{|u_{\infty}(X_{2})-u_{\infty}(Y_{2})|^{p-2}(u_{\infty}(X_{2})-u_{\infty}(Y_{2}))(v(x)-v(y))}{|x-y|^{N+sp}}\,dxdy.
		\end{align*}	
		This completes the proof.	
	\end{proof}		
	Next, in $\mathbb{R}$, we consider a function $\rho = \rho(x_{1})$ whose graph is shown in \textbf{Fig. 2}.
	\begin{center}
		\begin{tikzpicture}[
			scale=2,
			IS/.style={blue, thick},
			LM/.style={red, thick},
			axis/.style={very thick, ->, >=stealth', line join=miter},
			important line/.style={thick}, dashed line/.style={dashed, thin},
			every node/.style={color=black},
			dot/.style={circle,fill=black,minimum size=4pt,inner sep=0pt,
				outer sep=-1pt},
			]
			\coordinate (tint) at (2.5,0);
			\coordinate (beg_2) at (-1.8,0);
			\coordinate (beg_3) at (0.15,0);
			\coordinate (M) at (-1.1,0.45);
			\coordinate (M2) at (1.08,0.45);
			\coordinate (M3) at (-2.68,2.3);
			\coordinate (M4) at (2,2.3);
			\coordinate (t) at (1.5,0);
			\coordinate (t1) at (1,0);
			\coordinate (tt) at (-1.1,0);
			\coordinate (tt2) at (-0.62,0);
			\coordinate (M5) at (0,1.89);
			\coordinate (tf) at (0.2,1.02);
			\draw[black, thick] (-0.55,0.8)--(0,0.8)--(1,0.8);
			\draw[thick,->] (-2,0) -- (2.5,0) node[anchor=north west] {$x_{1}$};
			\draw[thick,->] (0.27,0) -- (0.27,1.5) node[anchor=south east] {$z_{1}$};

			\draw[black,thick](-1,0)--(-0.55,0.8);
			\draw[black, thick](1.5,0)--(1,0.8);
			\draw[black, dashed](1,0)--(1,0.8);
			\draw[black,dashed](-0.56,0)--(-0.55,0.8);

			\fill[black](beg_3)node[below]{$0$};
			
			\fill[black](tf)  node[below]{$1$};
			\fill[black](t)  node[below]{$1$};
			\fill[black](t1)  node[below]{$\frac{1}{2}$};
			\fill[black](tt)  node[below]{$-1$};
			\fill[black](tt2)  node[below]{$-\frac{1}{2}$};
		\end{tikzpicture}
	\end{center}
	\begin{center}
		\textbf{Fig.2.}
	\end{center}
	Clearly, this function satisfies
	\begin{equation}\label{FUN}
		0\leq \rho \leq 1, \;\;\rho =1 \; \text{on}\; \left(\frac{-1}{2}, \frac{1}{2}\right), \;\;\rho =0 \;\text{on }\;\mathbb{R}\backslash(-1,1),\;\;\;|\rho'|\leq 2.
	\end{equation}
	\begin{lemma}\label{LEV}
		Let $u_{\infty}$ be the unique weak solution of (\ref{EQQ1}) and $u_{\ell}$ be the unique weak solution to (\ref{EQ1bis}). Then, there holds,
		\begin{equation}
			(u_{\ell}-u_{\infty})\rho_{\ell}^{p}(x_{1})	\in W_{0}^{s,p}(\Omega_{\ell}).
		\end{equation}
		where $\rho_{\ell}(x_{1})=\rho \left(\frac{x_{1}}{\ell}\right)$. 	
	\end{lemma}
	\begin{proof}
		It is sufficient to prove that $\psi_{\ell}(x):=u_{\infty}(X_{2})\rho^{p} \left(\frac{x_{1}}{\ell}\right)\in W_{0}^{s,p}(\Omega_{\ell})$. The proof of $u_{\ell}(x)\rho^{p} \left(\frac{x_{1}}{\ell}\right)\in W_{0}^{s,p}(\Omega_{\ell})$  follows by the same arguments. Indeed, it is clear that $\psi_{\ell}\in L^{p}(\Omega_{\ell})$.  We now consider the Gagliardo seminorm of $\psi_{\ell}$:
		\begin{eqnarray*}
			[\psi_{\ell}]^{p}_{s,p, \mathbb{R}^{N}}&=&\int_{\mathbb{R}^{N}}\int_{\mathbb{R}^{N}} \frac{|\psi_{\ell}(x)-\psi_{\ell}(y)|^{p}}{|x-y|^{N+sp}}\,dxdy\\
			&=& \int_{\Omega_{\ell}}\int_{\mathbb{R}^{N}} \frac{|\psi_{\ell}(x)-\psi_{\ell}(y)|^{p}}{|x-y|^{N+sp}}\,dxdy+ \int_{\mathbb{R}^{N}\backslash\Omega_{\ell}}\int_{\mathbb{R}^{N}} \frac{|\psi_{\ell}(x)-\psi_{\ell}(y)|^{p}}{|x-y|^{N+sp}}\,dxdy\\
			&\leq & 2 \int_{\Omega_{\ell}}\int_{\mathbb{R}^{N}} \frac{|\psi_{\ell}(x)-\psi_{\ell}(y)|^{p}}{|x-y|^{N+sp}}\,dxdy+\underbrace{\int_{\mathbb{R}^{N}\backslash\Omega_{\ell}}\int_{\mathbb{R}^{N}\backslash\Omega_{\ell}} \frac{|\psi_{\ell}(x)-\psi_{\ell}(y)|^{p}}{|x-y|^{N+sp}}\,dxdy}_{=0}\\
			&=&I_{\ell}.
		\end{eqnarray*}
		Thus, it remains to show that $I_{\ell}<+\infty$. Using the inequality $(2)$ in Lemma \ref{ILM} and the symmetry of the integral in the Gagliardo norm in $x$ and $y$, we obtain
		\begin{eqnarray*}
			I_{\ell}=2\int_{\Omega_{\ell}}\int_{\mathbb{R}^{N}} \frac{|\psi_{\ell}(x)-\psi_{\ell}(y)|^{p}}{|x-y|^{N+sp}}\,dxdy&=& 2\int_{\Omega_{\ell}}\int_{\mathbb{R}^{N}} \frac{|\psi_{\ell}(x)-\psi_{\ell}(y)|^{p}}{|x-y|^{N+sp}}\,dydx\\
			&\leq &
			2^{p}\int_{\Omega_{\ell}}\int_{\mathbb{R}^{N}} \frac{|u_{\infty}(X_{2})|^{p}|\rho^{p}_{\ell}(x_{1})-\rho^{p}_{\ell}(y_{1})|^{p}}{|x-y|^{N+sp}}\,dydx\\
			&&+ 2^{p} \int_{\Omega_{\ell}}\int_{\mathbb{R}^{N}} \frac{\rho^{p^{2}}_{\ell}(y_{1})|u_{\infty}(X_{2})-u_{\infty}(Y_{2})|^{p}}{|x-y|^{N+sp}}\,dydx\\
			&=& I^{1}_{\ell}+I^{2}_{\ell}.
		\end{eqnarray*}
		\textbf{Estimate of $I^{2}_{\ell}$.} We first note that
		$$\int_{\Omega_{\ell}} \int_{\mathbb{R}^{N}} \frac{|u_{\infty}(X_{2})-u_{\infty}(Y_{2})|^{p}}{|x-y|^{N+sp}}\,dydx< +\infty.$$
		Indeed,
		\begin{multline*}
			\int_{\Omega_{\ell}} \int_{\mathbb{R}^{N}} \frac{|u_{\infty}(X_{2})-u_{\infty}(Y_{2})|^{p}}{|x-y|^{N+sp}}\,dxdy	=
			\int_{-\ell}^{\ell}\int_{\omega} \int_{\mathbb{R}^{N}} \frac{|u_{\infty}(X_{2})-u_{\infty}(Y_{2})|^{p}}{|x-y|^{N+sp}}\,dydX_{2}dx_{1}\\
			= \int_{-\ell}^{\ell}\int_{\omega} \int_{\mathbb{R}^{N}} \frac{|u_{\infty}(X_{2})-u_{\infty}(Y_{2})|^{p}}{|X_{2}-Y_{2}|^{N+sp}\left(1+\frac{|x_{1}-y_{1}|^{2}}{|X_{2}-Y_{2}|^{2}}\right)^{\frac{N+sp}{2}}}\,dydX_{2}dx_{1}\\
			=\int_{-\ell}^{\ell}\int_{\omega} \int_{\mathbb{R}^{N-1}} \frac{|u_{\infty}(X_{2})-u_{\infty}(Y_{2})|^{p}}{|X_{2}-Y_{2}|^{N+sp}}\left(\int_{\mathbb{R}}\frac{dy_{1}}{\left(1+\frac{|x_{1}-y_{1}|^{2}}{|X_{2}-Y_{2}|^{2}}\right)^{\frac{N+sp}{2}}}\right)\,dY_{2}dX_{2}dx_{1}.
		\end{multline*}
		Applying Lemma \ref{Lem}, we deduce
		\begin{eqnarray}\label{EST}
			\nonumber \int_{\Omega_{\ell}} \int_{\mathbb{R}^{N}} \frac{|u_{\infty}(X_{2})-u_{\infty}(Y_{2})|^{p}}{|x-y|^{N+sp}}\,dxdy &\leq&  \theta_{N,p}\int_{-\ell}^{\ell}\int_{\mathbb{R}^{N-1}} \int_{\mathbb{R}^{N-1}} \frac{|u_{\infty}(X_{2})-u_{\infty}(Y_{2})|^{p}}{|X_{2}-Y_{2}|^{N-1+sp}}\,dY_{2}dX_{2}dx_{1},\\
			&\leq &2\ell \theta_{N,p} [u_{\infty}]^{p}_{s,p, \mathbb{R}^{N-1}}.
		\end{eqnarray}
		where $\theta_{N,p}$ is given in Lemma \ref{Lem}. Since,  $0\leq \rho_{\ell} \leq 1$, we conclude
		\begin{equation}\label{ESA}
			I^{2}_{\ell}\leq 2^{p+1}\ell \theta_{N,p} [u_{\infty}]^{p}_{s,p, \mathbb{R}^{N-1}}< +\infty.
		\end{equation}
		\textbf{Estimate of $I^{1}_{\ell}$.}	By inequality $(4)$ in Lemma \ref{ILM}, we have
		\begin{equation}\label{IMII}
			\left|\rho^{p}_{\ell}(x_{1})-\rho^{p}_{\ell}(y_{1})\right|\leq p\left|\rho_{\ell}(x_{1})-\rho_{\ell}(y_{1})\right|\left[\rho^{p-1}_{\ell}(x_{1})+\rho^{p-1}_{\ell}(y_{1})\right], \;\;\;\forall x_{1},y_{1}\in \mathbb{R}.
		\end{equation}
		Using this and the fact that $0\leq \rho_{\ell}\leq 1$, we obtain
		\begin{align*}
			I^{1}_{\ell}\leq& 2^{p}p^{p}\int_{\Omega_{\ell}} \int_{\mathbb{R}^{N}} \frac{|u_{\infty}(X_{2})|^{p}\left|\rho_{\ell}(x_{1})-\rho_{\ell}(y_{1})\right|^{p}\left[\rho^{p-1}_{\ell}(x_{1})+\rho^{p-1}_{\ell}(y_{1})\right]^{p}}{|x-y|^{N+sp}}\,dydx\\ \leq& 2^{p+1}\int_{\Omega_{\ell}}\int_{|x-y|< 1} \frac{|u_{\infty}(X_{2})|^{p}\left|\rho_{\ell}(x_{1})-\rho_{\ell}(y_{1})\right|^{p}}{|x-y|^{N+sp}}\,dydx\\&+2^{p+1}\int_{\Omega_{\ell}}\int_{|x-y|\geq 1} \frac{|u_{\infty}(X_{2})|^{p}\left|\rho_{\ell}(x_{1})-\rho_{\ell}(y_{1})\right|^{p}}{|x-y|^{N+sp}}\,dydx\\
			\leq&\frac{C}{\ell^{p}}\int_{\Omega_{\ell}}\int_{|x-y|< 1} \frac{|u_{\infty}(X_{2})|^{p}|x-y|^{p}}{|x-y|^{N+sp}}\,dydx+ C \int_{\Omega_{\ell}} \int_{|x-y|\geq1}  \frac{|u_{\infty}(X_{2})|^{p}}{|x-y|^{N+sp}}\,dydx\\
			=& \frac{C}{\ell^{p-1}}\|u_{\infty}\|^{p}_{p,\omega}\int_{B_{1}(0)}\frac{1}{|z|^{N+sp-p}}\,dz+C\ell \|u_{\infty}\|^{p}_{p,\omega}\int_{\mathbb{R}^{N}\backslash B_{1}(0)}\frac{1}{|z|^{N+sp}}\,dz<  +\infty.
		\end{align*}
		Combining the above estimates for $I^{1}_{\ell}$ and $I^{2}_{\ell}$ yields the desired conclusion.
	\end{proof}
	\begin{lemma}\label{LKM}
		Let $u_{\infty}$ be the unique weak solution of (\ref{EQQ1}). Then 
		\begin{align*}
			&\int_{\mathbb{R}^{N}\times \mathbb{R}^{N}} \frac{|u_{\infty}(X_{2})-u_{\infty}(Y_{2})|^{p}\rho^{p}_{\ell}(x_{1})}{|x-y|^{N+sp}}\,dxdy\\&\precsim \int_{\mathbb{R}^{N}\times \mathbb{R}^{N}}\frac{|u_{\infty}(Y_{2})|^{p}|\rho_{\ell}(x_{1})-\rho_{\ell}(y_{1})|^{p}}{|x-y|^{N+sp}}\,dxdy +\ell \|f\|_{2,\omega}\|u_{\infty}\|_{p,\omega}.
		\end{align*}
	\end{lemma}	
	\begin{proof}
		From Lemma \ref{LEV}, we have $u_{\infty}\rho^{p}_{\ell} \in W_{0}^{s, p}(\Omega_{\ell})$. Choosing $v= u_{\infty}\rho^{p}_{\ell} $ in (\ref{DFFR}) yields
		\begin{equation}\label{DFFR1}
			\frac{C_{N,s,p}}{2}	\int_{\mathbb{R}^{N}\times \mathbb{R}^{N}}\frac{\varphi_{p}(u_{\infty}(X_{2})-u_{\infty}(Y_{2}))((u_{\infty}\rho^{p}_{\ell})(x)-(u_{\infty}\rho^{p}_{\ell})(y))}{|x-y|^{N+sp}}\,dxdy=\int_{\Omega_{\ell}}fu_{\infty}\rho^{p}_{\ell}\,dx.
		\end{equation}
		From this identity, we deduce
		\begin{align*}
			&\frac{C_{N,s,p}}{2}	\int_{\mathbb{R}^{N}\times \mathbb{R}^{N}}\frac{|u_{\infty}(X_{2})-u_{\infty}(Y_{2})|^{p}\rho^{p}_{\ell}(x_{1})}{|x-y|^{N+sp}}\,dxdy\\&
			\leq\int_{\Omega_{\ell}}fu_{\infty}\rho^{p}_{\ell}\,dx+	\frac{C_{N,s,p}}{2}	\int_{\mathbb{R}^{N}\times \mathbb{R}^{N}}\frac{\varphi_{p}(u_{\infty}(X_{2})-u_{\infty}(Y_{2}))|u_{\infty}(Y_{2})|\left(\rho^{p}_{\ell}(x_{1})-\rho^{p}_{\ell}(y_{1})\right)}{|x-y|^{N+sp}}\,dxdy.
		\end{align*}
		Applying (\ref{IMII}) to the last term, we get
		\begin{align*}
			&\frac{C_{N,s,p}}{2}	\int_{\mathbb{R}^{N}\times \mathbb{R}^{N}}\frac{|u_{\infty}(X_{2})-u_{\infty}(Y_{2})|^{p}\rho^{p}_{\ell}(x_{1})}{|x-y|^{N+sp}}\,dxdy\\
			\leq&\int_{\Omega_{\ell}}fu_{\infty}\rho^{p}_{\ell}\,dx\\&+	\frac{pC_{N,s,p}}{2}	\int_{\mathbb{R}^{N}\times \mathbb{R}^{N}}\frac{|u_{\infty}(X_{2})-u_{\infty}(Y_{2})|^{p-1}\left(\rho^{p-1}_{\ell}(x_{1})+\rho^{p-1}_{\ell}(y_{1})\right)|u_{\infty}(Y_{2})||\rho_{\ell}(x_{1})-\rho_{\ell}(y_{1})|}{|x-y|^{N+sp}}\,dxdy.
		\end{align*}
		By Young’s inequality and Hölder’s inequality, for any $\epsilon>0$,
		\begin{align*}
			&\frac{C_{N,s,p}}{2}	\int_{\mathbb{R}^{N}\times \mathbb{R}^{N}}\frac{|u_{\infty}(X_{2})-u_{\infty}(Y_{2})|^{p}\rho^{p}_{\ell}(x_{1})}{|x-y|^{N+sp}}\,dxdy\\
			\leq& \ell \|f\|_{2,\omega}\|u_{\infty}\|_{p,\omega}+	\epsilon	\int_{\mathbb{R}^{N}\times \mathbb{R}^{N}}\frac{|u_{\infty}(X_{2})-u_{\infty}(Y_{2})|^{p}\left(\rho^{p-1}_{\ell}(x_{1})+\rho^{p-1}_{\ell}(y_{1})\right)^{\frac{p}{p-1}}}{|x-y|^{N+sp}}\,dxdy\\
			&+C(\epsilon)\int_{\mathbb{R}^{N}\times \mathbb{R}^{N}}\frac{|u_{\infty}(Y_{2})|^{p}|\rho_{\ell}(x_{1})-\rho_{\ell}(y_{1})|^{p}}{|x-y|^{N+sp}}\,dxdy \\
			\leq &\ell \|f\|_{2,\omega}\|u_{\infty}\|_{p,\omega}+	2^{\frac{1}{p-1}}\epsilon	\int_{\mathbb{R}^{N}\times \mathbb{R}^{N}}\frac{|u_{\infty}(X_{2})-u_{\infty}(Y_{2})|^{p}\left(\rho^{p}_{\ell}(x_{1})+\rho^{p}_{\ell}(y_{1})\right)}{|x-y|^{N+sp}}\,dxdy\\
			&+C(\epsilon)\int_{\mathbb{R}^{N}\times \mathbb{R}^{N}}\frac{|u_{\infty}(Y_{2})|^{p}|\rho_{\ell}(x_{1})-\rho_{\ell}(y_{1})|^{p}}{|x-y|^{N+sp}}\,dxdy\\
			=& \ell \|f\|_{2,\omega}\|u_{\infty}\|_{p,\omega}+	2^{\frac{p}{p-1}}\epsilon	\int_{\mathbb{R}^{N}\times \mathbb{R}^{N}}\frac{|u_{\infty}(X_{2})-u_{\infty}(Y_{2})|^{p}\rho^{p}_{\ell}(x_{1})}{|x-y|^{N+sp}}\,dxdy\\
			&+C(\epsilon)\int_{\mathbb{R}^{N}\times \mathbb{R}^{N}}\frac{|u_{\infty}(Y_{2})|^{p}|\rho_{\ell}(x_{1})-\rho_{\ell}(y_{1})|^{p}}{|x-y|^{N+sp}}\,dxdy.
		\end{align*}
		Choosing $\epsilon>0$  small enough and absorbing the second term into the left-hand side gives the claimed estimate.
	\end{proof}
	\begin{proof} (Theorem \ref{THE})
		We are now in a position to conclude the proof of Theorem \ref{THE}. For brevity, we introduce the notations $v_{\ell}(x,y):=u_{\ell}(x)-u_{\ell}(y)$ and $v_{\infty}(x,y):=u_{\infty}(X_{2})-u_{\infty}(Y_{2})$. Subtracting identities (\ref{IDNT1}) and (\ref{DFFR}), we obtain
		\begin{equation}\label{RW}
			\frac{C_{N,s,p}}{2}\int_{\mathbb{R}^{N}\times \mathbb{R}^{N}}\frac{\left[\varphi_{p}(v_{\ell}(x,y))-\varphi_{p}(v_{\infty}(x,y))\right](w(x)-w(y))}{|x-y|^{N+sp}}\,dxdy=0,
		\end{equation}
		for each $w\in W_{0}^{s,p}(\Omega_{\ell})$.
		By Lemma \ref{LEV}, we have $(u_{\ell}-u_{\infty})\rho_{\ell}^{p}\in W_{0}^{s,p}(\Omega_{\ell})$. Choosing $w=(u_{\ell}-u_{\infty})\rho_{\ell}^{p}$ in (\ref{RW}) gives 
		\begin{equation}\label{RW1}
			\frac{C_{N,s,p}}{2}\int_{\mathbb{R}^{N}\times \mathbb{R}^{N}}\frac{\left[\varphi_{p}(v_{\ell}(x,y))-\varphi_{p}(v_{\infty}(x,y))\right]\left([(u_{\ell}-u_{\infty})\rho_{\ell}^{p}](x)-[(u_{\ell}-u_{\infty})\rho_{\ell}^{p}](y)\right)}{|x-y|^{N+sp}}\,dxdy=0.
		\end{equation}
		A straightforward rearrangement shows that
		\begin{multline}\label{GCCC}
			I:=\int_{\mathbb{R}^{N}\times \mathbb{R}^{N}}\frac{\left[\varphi_{p}(v_{\ell}(x,y))-\varphi_{p}(v_{\infty}(x,y))\right]\left(v_{\ell}(x,y)-v_{\infty}(x,y)\right)\rho^{p}_{\ell}(x_{1})}{|x-y|^{N+sp}}\,dxdy\\
			=-\int_{\mathbb{R}^{N}\times \mathbb{R}^{N}}\frac{\varphi_{p}(v_{\ell}(x,y))\left(u_{\ell}(y)-u_{\infty}(Y_{2})\right)\left[\rho^{p}_{\ell}(x_{1})-\rho^{p}_{\ell}(y_{1})\right]}{|x-y|^{N+sp}}\,dxdy\hskip 4cm\\
			+ \int_{\mathbb{R}^{N}\times \mathbb{R}^{N}}\frac{\varphi_{p}(v_{\infty}(x,y))\left(u_{\ell}(y)-u_{\infty}(Y_{2})\right)\left[\rho^{p}_{\ell}(x_{1})-\rho^{p}_{\ell}(y_{1})\right]}{|x-y|^{N+sp}}\,dxdy=I^{1}_{\ell}-I^{2}_{\ell}.\hskip 2cm
		\end{multline}
		Next, applying the inequality (\ref{IMII}) to the right hand side of (\ref{GCCC}), we deduce
		\begin{align*}
			&|I_{\ell}^{1}-I_{\ell}^{2}|\\&\leq p\int_{\mathbb{R}^{N}\times \mathbb{R}^{N}} \frac{|\varphi_{p}(v_{\ell}(x,y))-\varphi_{p}(v_{\infty}(x,y))||u_{\ell}(y)-u_{\infty}(Y_{2})||\rho_{\ell}(x_{1})-\rho_{\ell}(y_{1})||\rho^{p-1}_{\ell}(x_{1})+\rho^{p-1}_{\ell}(y_{1})|}{|x-y|^{N+sp}}\,dxdy.
		\end{align*}
		Using inequality  $(5)$ from Lemma \ref{ILM}, it follows that
		\begin{align*}
			&|I_{\ell}^{1}-I_{\ell}^{2}|\\\precsim & \int_{\mathbb{R}^{N}\times \mathbb{R}^{N}} |v_{\ell}(x,y)-v_{\infty}(x,y)||v_{\ell}(x,y)|^{p-2}|u_{\ell}(y)-u_{\infty}(Y_{2})||\rho_{\ell}(x_{1})-\rho_{\ell}(y_{1})||\rho^{p-1}_{\ell}(x_{1})+\rho^{p-1}_{\ell}(y_{1})|\,d\mu(x,y)\\
			+&\int_{\mathbb{R}^{N}\times \mathbb{R}^{N}} |v_{\ell}(x,y)-v_{\infty}(x,y)||v_{\infty}(x,y)|^{p-2}|u_{\ell}(y)-u_{\infty}(Y_{2})||\rho_{\ell}(x_{1})-\rho_{\ell}(y_{1})||\rho^{p-1}_{\ell}(x_{1})+\rho^{p-1}_{\ell}(y_{1})|\,d\mu(x,y)\\
			=&I^{3}_{\ell}+I^{4}_{\ell},\hskip 13cm
		\end{align*}
		where  $d\mu(x,y)=\frac{1}{|x-y|^{N+sp}}\,dxdy$. We now estimate $I^{3}_{\ell}$ and $I^{4}_{\ell}$ separately. Since, $\rho\geq 0$ we have
		$$ \rho^{p-1}_{\ell}(x_{1})+\rho^{p-1}_{\ell}(y_{1})\leq  \left[\rho^{p-2}_{\ell}(x_{1})+\rho^{p-2}_{\ell}(y_{1})\right]\left[\rho_{\ell}(x_{1})+\rho_{\ell}(y_{1})\right].$$
		Applying this estimate in $I^{3}_{\ell}$ and using the generalized Hölder inequality with exponents $p$, $p$, and $\frac{p}{p-2}$, we obtain
		\begin{multline*}
			I^{3}_{\ell}\precsim\left(\int_{\mathbb{R}^{N}\times \mathbb{R}^{N}}|v_{\ell}(x,y)-v_{\infty}(x,y)|^{p}\rho^{p}_{\ell}(x_{1})\,d\mu(x,y)\right)^{1/p} \left(\int_{\mathbb{R}^{N}\times \mathbb{R}^{N}}|v_{\ell}(x,y)|^{p}\rho^{p}_{\ell}(x_{1})\,d\mu(x,y)\right)^{(p-2)/p}\\
			\times\left(\int_{\mathbb{R}^{N}\times\mathbb{R}^{N} }|u_{\ell}(y)-u_{\infty}(Y_{2})|^{p}||\rho_{\ell}(x_{1})-\rho_{\ell}(y_{1})|^{p}\,d\mu(x,y)\right)^{1/p}.
		\end{multline*}
		Similarly, for $I_{\ell}^{4}$ we obtain 
		\begin{multline*}
			I^{4}_{\ell}\precsim\left(\int_{\mathbb{R}^{N}\times \mathbb{R}^{N}}|v_{\ell}(x,y)-v_{\infty}(x,y)|^{p}\rho^{p}_{\ell}(x_{1})\,d\mu(x,y)\right)^{1/p} \left(\int_{\mathbb{R}^{N}\times \mathbb{R}^{N}}|v_{\infty}(x,y)|^{p}\rho^{p}_{\ell}(x_{1})\,d\mu(x,y)\right)^{(p-2)/p}\\
			\times\left(\int_{\mathbb{R}^{N}\times\mathbb{R}^{N} }|u_{\ell}(y)-u_{\infty}(Y_{2})|^{p}||\rho_{\ell}(x_{1})-\rho_{\ell}(y_{1})|^{p}\,d\mu(x,y)\right)^{1/p},
		\end{multline*}
		Combining these inequalities yields
		\begin{multline}\label{FCC}
			|I^{1}_{\ell}-I^{2}_{\ell}|
			\precsim \left(\int_{\mathbb{R}^{N}\times \mathbb{R}^{N} }|v_{\ell}(x,y)-v_{\infty}(x,y)|^{p}\rho^{p}_{\ell}(x_{1})\,d\mu(x,y)\right)^{1/p}
			\left[\Lambda^{1}_{\ell}+\Lambda^{2}_{\ell}\right]\Lambda^{3}_{\ell}.\\
		\end{multline}
		where
		$$ \Lambda^{1}_{\ell}=\left(\int_{\mathbb{R}^{N}\times\mathbb{R}^{N} }|v_{\ell}(x,y)|^{p}\rho^{p}_{\ell}(x_{1})\,d\mu(x,y)\right)^{(p-2)/p}, \quad \Lambda^{2}_{\ell}=\left(\int_{\mathbb{R}^{N}\times \mathbb{R}^{N}}|v_{\infty}(x,y)|^{p}\rho^{p}_{\ell}(x_{1})\,d\mu(x,y)\right)^{(p-2)/p},$$
		and
		$$ \Lambda_{\ell}^{3}=\left(\int_{\mathbb{R}^{N}\times \mathbb{R}^{N}}|u_{\ell}(y)-u_{\infty}(Y_{2})|^{p}||\rho_{\ell}(x_{1})-\rho_{\ell}(y_{1})|^{p}\,d\mu(x,y)\right)^{1/p}.$$
		Next, using the inequality $(1) $ in Lemma \ref{ILM} to the left hand side of (\ref{GCCC}) and using (\ref{FCC}), we conclude 
		\begin{align*}
			&\int_{\mathbb{R}^{N}\times \mathbb{R}^{N}}|v_{\ell}(x,y)-v_{\infty}(x,y)|^{p}\rho^{p}_{\ell}(x_{1})\,d\mu(x,y)\\&\precsim \left(\int_{\mathbb{R}^{N}\times \mathbb{R}^{N} }|v_{\ell}(x,y)-v_{\infty}(x,y)|^{p}\rho^{p}_{\ell}(x_{1})\,d\mu(x,y)\right)^{1/p}\left[\Lambda^{1}_{\ell}+\Lambda^{2}_{\ell}\right]\Lambda^{3}_{\ell},
		\end{align*}
		which implies
		\begin{equation}\label{GCC}
			\left(\int_{\mathbb{R}^{N}\times \mathbb{R}^{N}}|v_{\ell}(x,y)-v_{\infty}(x,y)|^{p}\rho^{p}_{\ell}(x_{1})\,d\mu(x,y)\right)^{(p-1)/p}\precsim \left[\Lambda^{2}_{\ell}+\Lambda^{2}_{\ell}\right]\Lambda^{3}_{\ell}.
		\end{equation}
		Applying the fractional $p$-Poincar\'e inequality to the function $(u_{\ell}-u_{\infty})\rho_{\ell}$ yields
		\begin{eqnarray*}\label{eqH}
			\nonumber\int_{\Omega_{\ell}}|u_{\ell}-u_{\infty}|^{p}\rho^{p}(x_{1})\,dx &\leq  &  C_{P}\int_{\mathbb{R}^{N}\times \mathbb{R}^{N} }\frac{|\left[(u_{\ell}-u_{\infty})\rho_{\ell}\right](x)-\left[(u_{\ell}-u_{\infty})\rho_{\ell}\right](y)|^{p}}{|x-y|^{N+sp}}\,dxdy\\
			&\leq& 2^{p-1}C_{P}\int_{\mathbb{R}^{N}\times \mathbb{R}^{N}} \frac{|v_{\ell}(x,y)-v_{\infty}(x,y)|^{p}\rho^{p}_{\ell}(x_{1})}{|x-y|^{N+sp}}\,dxdy\\\nonumber
			&&+2^{p-1}C_{P}\int_{\mathbb{R}^{N}\times \mathbb{R}^{N}} \frac{|\rho_{\ell}(x_{1})-\rho_{\ell}(y_{1})|^{p}|u_{\ell}(y)-u_{\infty}(Y_{2})|^{p}}{|x-y|^{N+sp}}\,dxdy.
		\end{eqnarray*}
		Using inequality $(3)$ in Lemma \ref{ILM}  gives 
		\begin{multline*}
			\|\left(u_{\ell}-u_{\infty}\right)\rho_{\ell}\|^{p-1}_{p,\Omega_{\ell}} \precsim
			\left(\int_{\mathbb{R}^{N}\times \mathbb{R}^{N}} |v_{\ell}(x,y)-v_{\infty}(x,y)|^{p}\rho^{p}_{\ell}(x_{1})\,d\mu(x,y)\right)^{(p-1)/p}\\
			+\left(\int_{\mathbb{R}^{N}\times \mathbb{R}^{N}} |\rho_{\ell}(x_{1})-\rho_{\ell}(y_{1})|^{p}|u_{\ell}(y)-u_{\infty}(Y_{2})|^{p}\,d\mu(x,y)\right)^{(p-1)/p}.
		\end{multline*}
		Combining this with (\ref{GCC}) yields 
		\begin{equation}\label{FDS}
			\|\left(u_{\ell}-u_{\infty}\right)\rho_{\ell}\|^{p-1}_{p,\Omega_{\ell}} \precsim \left(\Lambda^{1}_{\ell}+\Lambda^{2}_{\ell}\right)\Lambda^{3}_{\ell}+ \left[\Lambda^{3}_{\ell}\right]^{p-1}.
		\end{equation}
		On the other hand, from Lemmas \ref{LEm} and \ref{LKM} we have
		$$ \Lambda^{1}_{\ell}+\Lambda^{2}_{\ell}\precsim \left(\int_{\mathbb{R}^{N}\times \mathbb{R}^{N} }\frac{|u_{\infty}(Y_{2})|^{p}|\rho_{\ell}(x_{1})-\rho_{\ell}(y_{1})|^{p}}{|x-y|^{N+sp}}\,dxdy\right)^{\frac{p-2}{p}} +\ell^{\frac{p-2}{p}} \|f\|^{\frac{p-2}{p}}_{2,\omega}\|u_{\infty}\|^{\frac{p-2}{p}}_{p,\omega}+ \ell^{\frac{p-2}{p}} \|f\|^{\frac{p-2}{p-1}}_{2,\omega}.$$
		Substituting this into (\ref{FDS}) gives
		\begin{align}\label{PK}
			&\|\left(u_{\ell}-u_{\infty}\right)\rho_{\ell}\|^{p-1}_{p,\Omega_{\ell}} \\\precsim&\left(\left(\int_{\mathbb{R}^{N}\times \mathbb{R}^{N}}|u_{\infty}(Y_{2})|^{p}|\rho_{\ell}(x_{1})-\rho_{\ell}(y_{1})|^{p}\,d\mu(x,y)\right)^{\frac{p-2}{p}} +\ell^{\frac{p-2}{p}} \|f\|^{\frac{p-2}{p}}_{2,\omega}\|u_{\infty}\|^{\frac{p-2}{p}}_{p,\omega}+ \ell^{\frac{p-2}{p}} \|f\|^{\frac{p-2}{p-1}}_{2,\omega}\right)\Lambda^{3}_{\ell}\nonumber\\&+ \left[\Lambda^{3}_{\ell}\right]^{p-1}\nonumber.
		\end{align}
		Next, consider the function
		\begin{equation}\label{GF}
			h_{\ell}(y_{1}):=\int_{\mathbb{R}} \frac{|\rho_{\ell}(x_{1})-\rho_{\ell}(y_{1})|^{p}}{|x_{1}-y_{1}|^{1+sp}}\,dx_{1}.
		\end{equation}
		We shall frequently use estimates for  $h_{\ell}(y_{1})$. Observe first that if $|y_{1}|\geq 2\ell$ then $\rho_{\ell}(y_{1})=0$ and hence
		$$  h_{\ell}(y_{1})=\int_{\mathbb{R}} \frac{|\rho_{\ell}(x_{1})|^{p}}{|x_{1}-y_{1}|^{1+sp}}\,dx_{1}\stackrel{\text{from}\, \textbf{Fig.2.}}{\leq  }\int_{-\ell}^{\ell} \frac{1}{|x_{1}-y_{1}|^{1+sp}}\,dx_{1}\leq \frac{1}{sp}\left(\frac{1}{|-\ell +y_{1}|^{sp}}+\frac{1}{|\ell +y_{1}|^{sp}}\right). $$
		If  $|y_{1}|<2\ell$, using $|\rho'| \leq 2$ we have
		\begin{eqnarray*}
			h_{\ell}(y_{1})&=&\int_{-\ell}^{\ell} \frac{|\rho_{\ell}(x_{1}+y_{1})-\rho_{\ell}(y_{1})|^{p}}{|x_{1}|^{1+sp}}\,dx_{1}+\int_{\mathbb{R}\backslash (-\ell, \ell)} \frac{|\rho_{\ell}(x_{1}+y_{1})-\rho_{\ell}(y_{1})|^{p}}{|x_{1}|^{1+sp}}\,dx_{1}\\
			&\leq & \frac{2}{\ell^{p}}\int_{-\ell}^{\ell}\frac{|x_{1}|^{p}}{|x_{1}|^{1+sp}}\,dx_{1}+2\int_{\mathbb{R} \backslash (-\ell, \ell)}\frac{1}{|x_{1}|^{1+sp}}\,dx_{1},\\
			&\precsim & \frac{1}{\ell^{sp}},
		\end{eqnarray*}
		Therefore,
		\begin{equation}\label{NC}
			h_{\ell}(y_{1})\precsim \left\{
			\begin{array}{lcc}
				\frac{1}{\ell^{sp}} & \text{if}& y_{1}\in (-2\ell, 2\ell), \\
				\frac{1}{s}\left(\frac{1}{|-\ell +y_{1}|^{sp}}+\frac{1}{|\ell +y_{1}|^{sp}}\right)  & \text{if}& y_{1}\in \mathbb{R}\backslash (-2\ell, 2\ell).
			\end{array}\right.
		\end{equation}
		By convexity of the function $z\mapsto |z|^{p}$, we infer 
		\begin{equation}\label{VC2}
			\left[\varLambda^{3}_{\ell}\right]^{p}\precsim \int_{\mathbb{R}^{N}\times\mathbb{R}^{N} }\frac{|u_{\ell}(y)|^{p}|\rho_{\ell}(x_{1})-\rho_{\ell}(y_{1})|^{p}}{|x-y|^{N+sp}}\,dxdy+\int_{\mathbb{R}^{N}\times \mathbb{R}^{N}}\frac{|u_{\infty}(Y_{2})|^{p}|\rho_{\ell}(x_{1})-\rho_{\ell}(y_{1})|^{p}}{|x-y|^{N+sp}}\,dxdy=J^{1}_{\ell}+J^{2}_{\ell}.
		\end{equation}
		In what follows we estimate the terms $ J^{1}_{\ell}$ and $J^{2}_{\ell}$ respectively. From Lemma \ref{Lem} we have
		\begin{equation}\label{HB}
			\int_{\mathbb{R}^{N-1}}\frac{1}{\left(1+\frac{|X_{2}-Y_{2}|^{2}}{|x_{1}-y_{1}|^{2}}\right)^{\frac{N+sp}{2}}}\,dX_{2}=|x_{1}-y_{1}|^{N-1}\prod_{k=2}^{N}\theta_{k,p}.
		\end{equation}
		Hence,
		\begin{multline*}
			J^{1}_{\ell}=\int_{\mathbb{R}^{N}\times \mathbb{R}^{N}}\frac{|u_{\ell}(y)|^{p}|\rho_{\ell}(x_{1})-\rho_{\ell}(y_{1})|^{p}}{|x-y|^{N+sp}}\,dxdy\\
			=\int_{\mathbb{R}^{N}} |u_{\ell}(y)|^{p}\int_{\mathbb{R}} \frac{|\rho_{\ell}(x_{1})-\rho_{\ell}(y_{1})|^{p}}{|x_{1}-y_{1}|^{N+sp}}\int_{\mathbb{R}^{N-1}}\frac{1}{\left(1+\frac{|X_{2}-Y_{2}|^{2}}{|x_{1}-y_{1}|^{2}}\right)^{\frac{N+sp}{2}}}\,dX_{2}dx_{1}dy\hskip 1.9cm\\
			=\prod_{k=2}^{N}\theta_{k,p}\int_{\mathbb{R}^{N}}|u_{\ell}(y)|^{p}\int_{\mathbb{R}} \frac{|\rho_{\ell}(x_{1})-\rho_{\ell}(y_{1})|^{p}}{|x_{1}-y_{1}|^{1+sp}}\,dx_{1}dy\hskip 6cm\\
			=\prod_{k=2}^{N}\theta_{k,p}\int_{\mathbb{R}^{N}}|u_{\ell}(y)|^{p} h_{\ell}(y_{1})dy.\hskip 9.5cm
		\end{multline*}
		Since $u_{\ell}(y)=0$ in $\mathbb{R}^{N}\backslash \Omega_{\ell}$, it follows from (\ref{NC}) that
		\begin{equation*}\label{VC1}
			J_{\ell}^{1}=\prod_{k=2}^{N}\theta_{k,p}\int_{\mathbb{R}^{N}}|u_{\ell}(y)|^{p} h_{\ell}(y_{1})dy=\prod_{k=2}^{N}\theta_{k,p}\int_{\Omega_{\ell}}|u_{\ell}(y)|^{p} h_{\ell}(y_{1})dy\precsim \frac{1}{\ell^{sp}}\int_{\Omega_{\ell}}|u_{\ell}(y)|^{p}dy.
		\end{equation*}
		By Lemma \ref{LEm} this yields 
		\begin{equation}\label{KJ}
			J_{\ell}^{1}\precsim \frac{\|f\|^{\frac{p}{p-1}}_{2,\omega}}{\ell^{sp-1}}.
		\end{equation}
		Next, we note that since $s>\frac{p-1}{p}$ and $u_{\ell}(Y_{2})=0$ in $\mathbb{R}^{N-1}\backslash \omega$, it follows from (\ref{HB}) that
		\begin{multline}\label{EW1}
			J^{2}_{\ell}=\int_{\mathbb{R}^{N}\times \mathbb{R}^{N}}\frac{|u_{\infty}(Y_{2})|^{p}|\rho_{\ell}(x_{1})-\rho_{\ell}(y_{1})|^{p}}{|x-y|^{N+sp}}\,dxdy=\prod_{k=2}^{N}\theta_{k,p}\int_{\mathbb{R}^{N}}|u_{\infty}(Y_{2})|^{p}h_{\ell}(y_{1})dy\\=\prod_{k=2}^{N}\theta_{k,p}\int_{\mathbb{R}^{N-1}}|u_{\infty}(Y_{2})|^{p}\,dY_{2}\int_{-\infty}^{+\infty}h_{\ell}(y_{1})\,dy_{1}\hskip 7cm\\	
			= \prod_{k=2}^{N}\theta_{k,p}\int_{\omega}|u_{\infty}(Y_{2})|^{p}\,dY_{2}\left\{\int_{-2\ell}^{2\ell}h_{\ell}(y_{1})\,dy_{1}+\int_{2\ell}^{+\infty}h_{\ell}(y_{1})\,dy_{1}+\int_{-\infty}^{-2\ell}h_{\ell}(y_{1})\,dy_{1}\right\}\hskip 1cm\\\nonumber
			\precsim\frac{\|u_{\infty}\|^{p}_{p,\omega}}{\ell^{sp-1}}.\hskip 12.5cm
		\end{multline}
		Combining this with (\ref{KJ}) and (\ref{VC2}), we obtain
		$$\left[\Lambda^{3}_{\ell}\right]^{p}\precsim  \frac{\|f\|^{\frac{p}{p-1}}_{2,\omega}+\|u_{\infty}\|^{p}_{p,\omega} }{\ell^{sp-1}},$$
		which implies
		$$\Lambda_{3}\precsim  \frac{\|f\|^{\frac{1}{p-1}}_{2,\omega}+\|u_{\infty}\|_{p,\omega} }{\ell^{s-\frac{1}{p}}},$$
		Substituting this into (\ref{PK}) gives
		\begin{multline*}
			\|\left(u_{\ell}-u_{\infty}\right)\rho_{\ell}\|^{p-1}_{p,\Omega_{\ell}}\\\precsim \left(\frac{\|u_{\infty}\|^{p-2}_{p,\omega}}{\ell^{s(p-2)-\frac{p-2}{p}}} +\ell^{\frac{p-2}{p}} \|f\|^{\frac{p-2}{p}}_{2,\omega}\|u_{\infty}\|^{\frac{p-2}{p}}_{p,\omega}+ \ell^{\frac{p-2}{p}} \|f\|^{\frac{p-2}{p-1}}_{2,\omega}\right)\frac{\|f\|^{\frac{p}{p-1}}_{2,\omega}+\|u_{\infty}\|^{p}_{p,\omega} }{\ell^{s-\frac{1}{p}}}+ \frac{\|f\|^{\frac{p}{p-1}}_{2,\omega}+\|u_{\infty}\|^{p}_{p,\omega} }{\ell^{s(p-1)-\frac{p-1}{p}}}\\
			\precsim \frac{1}{\ell^{s(p-1)-\frac{p-1}{p}}}+\frac{1}{\ell^{s-\frac{p-1}{p}}}.\hskip 12cm
		\end{multline*}
		Thus,
		$$ \|\left(u_{\ell}-u_{\infty}\right)\rho_{\ell}\|_{p,\Omega_{\ell}}\precsim \left[\frac{1}{\ell^{s(p-1)-\frac{p-1}{p}}}+\frac{1}{\ell^{s-\frac{p-1}{p}}}\right]^{\frac{1}{p-1}}.$$
		Since $\rho_{\ell}=1$ on $\Omega_{\ell/2}$,  it follows that
		\begin{equation}
			\|u_{\ell}-u_{\infty}\|_{p, \Omega_{\ell/2}}\precsim \left[\frac{1}{\ell^{s(p-1)-\frac{p-1}{p}}}+\frac{1}{\ell^{s-\frac{p-1}{p}}}\right]^{\frac{1}{p-1}}.
		\end{equation}
		Choosing $\ell$ large enough so that $\ell/2 > \ell_{0}$, we obtain
		\begin{equation*}
			\|u_{\ell}-u_{\infty}\|_{p, \Omega_{\ell_{0}}}\precsim \left[\frac{1}{\ell^{s(p-1)-\frac{p-1}{p}}}+\frac{1}{\ell^{s-\frac{p-1}{p}}}\right]^{\frac{1}{p-1}}.	
		\end{equation*}
		This completes the proof.
	\end{proof}
	\section{Proof of Theorem \ref{THP}}
	This section is devoted to proving Theorem \ref{THP}. To that end, we first prepare several lemmas that will be utilized in the sequel.
	\begin{lemma}\label{LEm1}
		Let $u_{\ell}$ be the weak solution to (\ref{EQ2}). Then, 
		\begin{equation}
			\|u_{\ell}\|_{L^{\infty}(0,T; L^{2}(\Omega_{\ell}))}^{2}	+\|u_{\ell}\|^{p}_{L^{p}(0,T;L^{p}(\Omega_{\ell}))}\precsim  \ell \left\{\|f\|_{L^{2}(0,T;L^{2}(\omega))}^{\frac{p}{p-1}}+\|u_{0}\|^{2}_{2,\omega}\right\}.
		\end{equation}
	\end{lemma}
	\begin{proof}
		Choosing $v=u_{\ell}$ in (\ref{IDNT2}) and applying Hölder’s inequality together with Proposition \ref{Pr1}, we have
		$$\frac{1}{2} \frac{d}{dt} \|u_{\ell}(t)\|^{2}_{2,\Omega_{\ell}}+\frac{C_{N,s,p}}{2} \int_{\mathbb{R}^{N}\times \mathbb{R}^{N}} \frac{|u_{\ell}(x,t)-u_{\ell}(y,t)|^{p}}{|x-y|^{N+ps}}\,dxdy\leq\|u_{\ell}\|_{p,\Omega_{\ell}}\left(\int_{-\ell}^{\ell}\int_{\omega}|f(X_{2},t)|^{p'}\,dx\right)^{\frac{1}{p'}},$$
		where $p'=\frac{p}{p-1}$.
		Using the fractional $p$-Poincar\'e inequality in the right hand side gives 
		\begin{multline}\label{GV}
			\frac{1}{2} \frac{d}{dt} \|u_{\ell}(t)\|^{2}_{2,\Omega_{\ell}}+	\frac{C_{N,s,p}}{2}\int_{\mathbb{R}^{N}\times \mathbb{R}^{N}}  \frac{|u_{\ell}(x,t)-u_{\ell}(y,t)|^{p}}{|x-y|^{N+ps}}\,dxdy\\\leq C_{P} \left(\int_{\mathbb{R}^{N}\times \mathbb{R}^{N}}  \frac{|u_{\ell}(x,t)-u_{\ell}(y,t)|^{p}}{|x-y|^{N+ps}}\,dxdy\right)^{\frac{1}{p}}\left(\int_{-\ell}^{\ell}\int_{\omega}|f(X_{2},t)|^{p'}\,dX_{2}dx_{1}\right)^{\frac{1}{p'}},
		\end{multline}
		where $C_{P}>0$  is the best constant in the fractional $p$-Poincar\'e inequality. Applying Young’s inequality to the right-hand side  of (\ref{GV})  yields
		\begin{multline*}\label{GVb}
			\frac{1}{2} \frac{d}{dt} \|u_{\ell}(t)\|^{2}_{2,\Omega_{\ell}}+	\frac{C_{N,s,p}}{2}\int_{\mathbb{R}^{N}\times \mathbb{R}^{N}}  \frac{|u_{\ell}(x,t)-u_{\ell}(y,t)|^{p}}{|x-y|^{N+ps}}\,dxdy\\\leq \epsilon \int_{\mathbb{R}^{N}\times \mathbb{R}^{N}}  \frac{|u_{\ell}(x,t)-u_{\ell}(y,t)|^{p}}{|x-y|^{N+ps}}\,dxdy+2C_{\epsilon}\ell\int_{\omega}|f(X_{2},t)|^{p'}\,dX_{2},\;\;\;\text{for all}\;\epsilon >0.
		\end{multline*}
		Choosing $\epsilon=\frac{C_{N,s,p}}{4}$ gives 
		\begin{equation}\label{HBV}
			\frac{1}{2} \frac{d}{dt} \|u_{\ell}(t)\|^{2}_{2,\Omega_{\ell}}+	\frac{C_{N,s,p}}{4}\int_{\mathbb{R}^{N}\times \mathbb{R}^{N}}  \frac{|u_{\ell}(x,t)-u_{\ell}(y,t)|^{p}}{|x-y|^{N+ps}}\,dxdy\leq C\ell\int_{\omega}|f(X_{2},t)|^{p'}\,dX_{2},
		\end{equation}
		Since $p'<2$,  Hölder’s inequality and integration over 
		$(0,t)$ yield
		\begin{equation}\label{HBV2}
			\frac{1}{2}  \|u_{\ell}(t)\|^{2}_{2,\Omega_{\ell}}+	\frac{C_{N,s,p}}{4}\int_{0}^{t}\int_{\mathbb{R}^{N}\times \mathbb{R}^{N}}  \frac{|u_{\ell}(x,\sigma)-u_{\ell}(y,\sigma)|^{p}}{|x-y|^{N+ps}}\,dxdyd\sigma\leq C\ell \left\{\|f\|^{p'}_{L^{2}(0,T;L^{2}(\omega))}+\|u_{0}\|^{2}_{2}\right\},
		\end{equation}
		for all $t\in [0,T].$ Further, using the fractional $p$-Poincar\'e inequality in the left hand side of (\ref{HBV2}) gives 
		\begin{equation}\label{HBV3}
			\|u_{\ell}\|^{2}_{L^{\infty}(0,T; L^{2}(\Omega_{\ell}))}+	\|u_{\ell}\|^{p}_{L^{p}(0,T;L^{2}(\Omega_{\ell}))}\precsim \ell \left\{\|f\|^{p'}_{L^{2}(0,T;L^{2}(\omega))}+\|u_{0}\|^{2}_{2}\right\}.
		\end{equation}
		Hence, this completes the proof.
	\end{proof}
	The proof of the next lemma can be carried out in a similar manner to that of Lemma \ref{LE}.
	\begin{lemma}\label{LEE}
		Let $x=(x_{1}, X_{2})\in \Omega_{\ell}$ and let $u_{\infty}$ be the unique weak solution to (\ref{eq3}) . Then, for any $v\in W_{0}^{s,p}(\Omega_{\ell})$, there holds
		\begin{equation}\label{DFFX}
			\int_{\Omega_{\ell}}\partial_{t}u_{\infty}(t)v\,dx+	\frac{C_{N,s}}{2}	\int_{\mathbb{R}^{N}\times \mathbb{R}^{N} }\frac{\varphi_{p}(u_{\infty}(X_{2},t)-u_{\infty}(Y_{2},t))(v(x)-v(y))}{|x-y|^{N+sp}}\,dxdy=\int_{\Omega_{\ell}}f(t)v\,dx, 
		\end{equation}
		for a.e. $t\in (0,T).$
	\end{lemma}

	\begin{lemma}\label{LKM1}
		Let $u_{\infty}$ be the unique weak solution of (\ref{eq3}). Then the following estimate holds:
		\begin{multline*}
			\|u_{\infty}\sqrt{\rho^{p}_{\ell}}\|^{2}_{2,\Omega_{\ell}}+	\int_{0}^{t}\int_{\mathbb{R}^{N}\times \mathbb{R}^{N}} \frac{|u_{\infty}(X_{2},\sigma)-u_{\infty}(Y_{2},\sigma)|^{p}\rho^{p}_{\ell}(x_{1})}{|x-y|^{N+sp}}\,dxdyd\sigma\\\precsim \int_{0}^{t}\int_{\mathbb{R}^{N}\times \mathbb{R}^{N}}\frac{|u_{\infty}(Y_{2},\sigma)|^{p}|\rho_{\ell}(x_{1})-\rho_{\ell}(y_{1})|^{p}}{|x-y|^{N+sp}}\,dxdyd\sigma +\ell \|f\|_{L^{2}(0,T; L^{2}(\omega))}\|u_{\infty}\|_{L^{p}(0,T;L^{p}(\omega))}+\ell \|u_{0}\|^{2}_{2,\omega}.
		\end{multline*}
	\end{lemma}	
	
	\begin{proof}
		The proof is similar to that of Lemma \ref{LKM1}, so we only sketch the main steps.	By Lemma \ref{LEV}, we have  $u_{\infty}\rho^{p}_{\ell} \in W_{0}^{s, p}(\Omega_{\ell})$. Taking $v= u_{\infty}\rho^{p}_{\ell} $ in (\ref{DFFX}), we obtain 
		\begin{multline}\label{DFFR2}
			\int_{\Omega_{\ell}} \partial_{t} u_{\infty}u_{\infty}\rho^{p}_{\ell}(x_{1})\,dx	\frac{C_{N,s}}{2}	\int_{\mathbb{R}^{N}\times \mathbb{R}^{N}}\frac{\varphi_{p}(u_{\infty}(X_{2},t)-u_{\infty}(Y_{2},t))((u_{\infty}\rho^{p}_{\ell})(x)-(u_{\infty}\rho^{p}_{\ell})(y))}{|x-y|^{N+sp}}\,dxdy\\
			=\int_{\Omega_{\ell}}fu_{\infty}\rho^{p}_{\ell}\,dx.
		\end{multline}
		Proceeding as in Lemma \ref{LKM1}, we arrive at
		\begin{multline*}
			\int_{\Omega_{\ell}} \partial_{t} u_{\infty}u_{\infty}\rho^{p}_{\ell}(x_{1})\,dx+	\frac{C_{N,s}}{2}	\int_{\mathbb{R}^{N}\times \mathbb{R}^{N}}\frac{|u_{\infty}(X_{2},t)-u_{\infty}(Y_{2},t)|^{p}\rho^{p}_{\ell}(x_{1})}{|x-y|^{N+sp}}\,dxdy\\
			\leq \ell \|f\|_{2,\omega}\|u_{\infty}\|_{p,\omega}+	2^{\frac{p}{p-1}}\epsilon	\int_{\mathbb{R}^{N}\times \mathbb{R}^{N}}\frac{|u_{\infty}(X_{2},t)-u_{\infty}(Y_{2},t)|^{p}\rho^{p}_{\ell}(x_{1})}{|x-y|^{N+sp}}\,dxdy\hskip 3cm\\
			+C(\epsilon)\int_{\mathbb{R}^{N}\times \mathbb{R}^{N}}\frac{|u_{\infty}(Y_{2},t)|^{p}|\rho_{\ell}(x_{1})-\rho_{\ell}(y_{1})|^{p}}{|x-y|^{N+sp}}\,dxdy, \;\;\;\text{for all}\;\epsilon >0.\hskip 3cm
		\end{multline*}
		Integrating this over \((0,t)\) and applying Proposition \ref{Pr1} together with Hölder’s inequality, we find
		\begin{align*}
			&\frac{1}{2}	\|u_{\infty}\sqrt{\rho^{p}_{\ell}}\|^{2}_{2,\Omega_{\ell}}+	\frac{C_{N,s}}{2}	\int_{0}^{t}\int_{\mathbb{R}^{N}\times \mathbb{R}^{N}}\frac{|u_{\infty}(X_{2},\sigma)-u_{\infty}(Y_{2},\sigma)|^{p}\rho^{p}_{\ell}(x_{1})}{|x-y|^{N+sp}}\,dxdyd\sigma\\\leq & \ell \|f\|_{L^{2}(0,T; L^{2}(\omega))}\|u_{\infty}\|_{L^{p}(0,T;L^{p}(\omega))}+\ell \|u_{0}\|^{2}_{2,\omega}\\&+	2^{\frac{p}{p-1}}\epsilon	\int_{0}^{t}\int_{\mathbb{R}^{N}\times \mathbb{R}^{N}}\frac{|u_{\infty}(X_{2},\sigma)-u_{\infty}(Y_{2},\sigma)|^{p}\rho^{p}_{\ell}(x_{1})}{|x-y|^{N+sp}}\,dxdyd\sigma\\
			&+C(\epsilon)\int_{0}^{t}\int_{\mathbb{R}^{N}\times \mathbb{R}^{N}}\frac{|u_{\infty}(Y_{2},\sigma)|^{p}|\rho_{\ell}(x_{1})-\rho_{\ell}(y_{1})|^{p}}{|x-y|^{N+sp}}\,dxdyd\sigma, \;\;\;\text{for all}\; \epsilon >0.
		\end{align*}
		Finally, choosing \(\epsilon>0\) sufficiently small and absorbing the corresponding term on the left-hand side yields the desired estimate.
	\end{proof}
	\begin{proof}(Theorem \ref{THP})
		Now we are in a position to finish the proof of Theorem \ref{THP}. To this end, we set
		
		$v_{\ell}(x,y,t)=u_{\ell}(x,t)-u_{\ell}(y,t)$ and $v_{\infty}(x,y,t)=u_{\ell}(X_{2},t)-u_{\ell}(Y_{2},t)$. Next, subtracting \eqref{IDNT2} from \eqref{DFFX} yields
		\begin{equation}\label{DFFX1}
			\int_{\Omega_{\ell}}\partial_{t}\left(u_{\ell}(t)-u_{\infty}(t)\right)v\,dx	+\frac{C_{N,s}}{2}	\int_{\mathbb{R}^{N}\times \mathbb{R}^{N} }\left[\varphi_{p}(v_{\ell}(x,,y,t))-\varphi_{p}(v_{\infty}(x,y,t))\right](v(x)-v(y))\,d\mu(x,y)=0,
		\end{equation}
		for any $v\in W_{0}^{s,p}(\Omega_{\ell})$ and	a.e. $t\in (0,T)$.  From Lemma \ref{LEV}, we know that $\left(u_{\ell}-u_{\infty}\right)\rho^{p}_{\ell} \in W_{0}^{s, p}(\Omega_{\ell})$. Thus, taking $v=\left(u_{\ell}-u_{\infty}\right)\rho^{p}_{\ell}$ in (\ref{DFFX1}) yields
		\begin{multline*}
			\int_{\Omega_{\ell}}\partial_{t}\left(u_{\ell}(t)-u_{\infty}(t)\right)\left[u_{\ell}-u_{\infty}\right]\rho^{p}_{\ell}(x_{1})\,dx	\\
			+\frac{C_{N,s}}{2}	\int_{\mathbb{R}^{N}\times \mathbb{R}^{N} }\left[\varphi_{p}(v_{\ell}(x,,y,t))-\varphi_{p}(v_{\infty}(x,y,t))\right](\left(\left[u_{\ell}-u_{\infty}\right]\rho^{p}_{\ell}\right)(x)-\left(\left[u_{\ell}-u_{\infty}\right]\rho^{p}_{\ell}\right)(y))\,d\mu(x,y)=0.
		\end{multline*}
		Using the same argument as in the previous section, we find
		\begin{multline}\label{GVZ}
			\int_{\Omega_{\ell}}\partial_{t}\left(u_{\ell}(t)-u_{\infty}(t)\right)\left[u_{\ell}-u_{\infty}\right]\rho^{p}_{\ell}(x_{1})\,dx	+\int_{\mathbb{R}^{N}\times \mathbb{R}^{N}}|v_{\ell}(x,y,t)-v_{\infty}(x,y,t)|^{p}\rho^{p}_{\ell}(x_{1})\,d\mu(x,y)\\
			\precsim \left(\int_{\mathbb{R}^{N}\times \mathbb{R}^{N} }|v_{\ell}(x,y,t)-v_{\infty}(x,y,t)|^{p}\rho^{p}_{\ell}(x_{1})\,d\mu(x,y)\right)^{1/p}
			\left[\Lambda^{1}_{\ell}(t)+\Lambda^{\ell}_{2}(t)\right]\Lambda^{3}_{\ell}(t),
		\end{multline}	
		where
		$$ \Lambda^{1}_{\ell}(t)=\left(\int_{\mathbb{R}^{N}\times\mathbb{R}^{N} }|v_{\ell}(x,y,t)|^{p}\rho^{p}_{\ell}(x_{1})\,d\mu(x,y)\right)^{(p-2)/p},$$ $$ \Lambda^{2}_{\ell}(t)=\left(\int_{\mathbb{R}^{N}\times \mathbb{R}^{N}}|v_{\infty}(x,y,t)|^{p}\rho^{p}_{\ell}(x_{1})\,d\mu(x,y)\right)^{(p-2)/p},$$
		and
		$$ \Lambda_{\ell}^{3}(t)=\left(\int_{\mathbb{R}^{N}\times \mathbb{R}^{N}}|u_{\ell}(y,t)-u_{\infty}(Y_{2},t)|^{p}||\rho_{\ell}(x_{1})-\rho_{\ell}(y_{1})|^{p}\,d\mu(x,y)\right)^{1/p}.$$
		Next, applying Young’s inequality to the right-hand side of \eqref{GVZ} gives
		\begin{multline*}
			\int_{\Omega_{\ell}}\partial_{t}\left(u_{\ell}(t)-u_{\infty}(t)\right)\left[u_{\ell}-u_{\infty}\right]\rho^{p}_{\ell}(x_{1})\,dx	+\int_{\mathbb{R}^{N}\times \mathbb{R}^{N}}|v_{\ell}(x,y,t)-v_{\infty}(x,y,t)|^{p}\rho^{p}_{\ell}(x_{1})\,d\mu(x,y)\\
			\precsim \epsilon\int_{\mathbb{R}^{N}\times \mathbb{R}^{N} }|v_{\ell}(x,y,t)-v_{\infty}(x,y,t)|^{p}\rho^{p}_{\ell}(x_{1})\,d\mu(x,y)+C(\epsilon)
			\left[\Lambda^{1}_{\ell}(t)+\Lambda^{\ell}_{2}(t)\right]^{\frac{p}{p-1}}\left[\Lambda^{3}_{\ell}(t)\right]^{\frac{p}{p-1}}, \;\;\;\forall \epsilon>0.
		\end{multline*}
		Taking $\epsilon$ small enough, we deduce
		\begin{multline*}\label{GVZfs}
			\int_{\Omega_{\ell}}\partial_{t}\left(u_{\ell}(t)-u_{\infty}(t)\right)\left[u_{\ell}-u_{\infty}\right]\rho^{p}_{\ell}(x_{1})\,dx	+\int_{\mathbb{R}^{N}\times \mathbb{R}^{N}}|v_{\ell}(x,y,t)-v_{\infty}(x,y,t)|^{p}\rho^{p}_{\ell}(x_{1})\,d\mu(x,y)\\
			\precsim
			\left[\Lambda^{1}_{\ell}(t)+\Lambda^{\ell}_{2}(t)\right]^{\frac{p}{p-1}}\left[\Lambda^{3}_{\ell}(t)\right]^{\frac{p}{p-1}}.
		\end{multline*}
		Moreover, in view of the inequality  $(2)$  in Lemma \ref{ILM}, we obtain
		\begin{multline*}
			\int_{\Omega_{\ell}}\partial_{t}\left(u_{\ell}(t)-u_{\infty}(t)\right)\left[u_{\ell}-u_{\infty}\right]\rho^{p}_{\ell}(x_{1})\,dx	+\int_{\mathbb{R}^{N}\times \mathbb{R}^{N}}|v_{\ell}(x,y,t)-v_{\infty}(x,y,t)|^{p}\rho^{p}_{\ell}(x_{1})\,d\mu(x,y)\\
			\precsim
			\left(\left[\Lambda^{1}_{\ell}(t)\right]^{\frac{p}{p-1}}+\left[\Lambda^{\ell}_{2}(t)\right]^{\frac{p}{p-1}}\right)\left[\Lambda^{3}_{\ell}(t)\right]^{\frac{p}{p-1}}.
		\end{multline*}
		Further, by Proposition \ref{Pr1} we infer that

		$$ \int_{0}^{t}	\int_{\Omega_{\ell}}\partial_{\sigma}\left(u_{\ell}(\sigma)-u_{\infty}(\sigma)\right)\left[u_{\ell}(\sigma)-u_{\infty}(\sigma)\right]\rho^{p}_{\ell}(x_{1})\,dxd\sigma =\frac{1}{2}\|\left[u_{\ell}-u_{\infty}\right](t)\sqrt{\rho_{\ell}^{p}}\|^{2}_{2,\Omega_{\ell}}.$$
		Here we used the fact that $u_{\ell}(x,0)=u_{\infty}(X_{2}, 0)=u_{0}(X_{2})$.  Using this and integrating the previous inequality over $(0,t) $, we obtain 
		\begin{multline*}
			\frac{1}{2}\|\left[u_{\ell}-u_{\infty}\right](t)\sqrt{\rho_{\ell}^{p}}\|^{2}_{2,\Omega_{\ell}}	+\int_{0}^{t}\int_{\mathbb{R}^{N}\times \mathbb{R}^{N}}|v_{\ell}(x,y,\sigma)-v_{\infty}(x,y,\sigma)|^{p}\rho^{p}_{\ell}(x_{1})\,d\mu(x,y)d\sigma\\
			\precsim
			\int_{0}^{t}\left(\left[\Lambda^{1}_{\ell}(\sigma)\right]^{\frac{p}{p-1}}+\left[\Lambda^{\ell}_{2}(\sigma)\right]^{\frac{p}{p-1}}\right)\left[\Lambda^{3}_{\ell}(\sigma)\right]^{\frac{p}{p-1}}\,d\sigma.
		\end{multline*}
		Recall the definitions
		$$ \left[\Lambda^{1}_{\ell}(t)\right]^{\frac{p}{p-1}}=\left(\int_{\mathbb{R}^{N}\times\mathbb{R}^{N} }|v_{\ell}(x,y,t)|^{p}\rho^{p}_{\ell}(x_{1})\,d\mu(x,y)\right)^{\frac{p-2}{p-1}}=\left[\Sigma_{\ell}^{1}(t)\right]^{\frac{p-2}{p-1}},$$
		$$ \left[\Lambda^{2}_{\ell}(t)\right]^{\frac{p}{p-1}}=\left(\int_{\mathbb{R}^{N}\times \mathbb{R}^{N}}|v_{\infty}(x,y,t)|^{p}\rho^{p}_{\ell}(x_{1})\,d\mu(x,y)\right)^{\frac{p-2}{p-1}}=\left[\Sigma_{\ell}^{2}(t)\right]^{\frac{p-2}{p-1}},$$
		and
		$$ \left[\Lambda_{\ell}^{3}(t)\right]^{\frac{p}{p-1}}=\left(\int_{\mathbb{R}^{N}\times \mathbb{R}^{N}}|u_{\ell}(y,t)-u_{\infty}(Y_{2},t)|^{p}||\rho_{\ell}(x_{1})-\rho_{\ell}(y_{1})|^{p}\,d\mu(x,y)\right)^{\frac{1}{p-1}}=\left[\Sigma_{\ell}^{3}(t)\right]^{\frac{1}{p-1}}.$$
		Substituting these into the last inequality yields
		\begin{multline*}
			\frac{1}{2}\|\left[u_{\ell}-u_{\infty}\right](t)\sqrt{\rho_{\ell}^{p}}\|^{2}_{2,\Omega_{\ell}}	+\int_{0}^{t}\int_{\mathbb{R}^{N}\times \mathbb{R}^{N}}|v_{\ell}(x,y,\sigma)-v_{\infty}(x,y,\sigma)|^{p}\rho^{p}_{\ell}(x_{1})\,d\mu(x,y)d\sigma\\
			\precsim
			\int_{0}^{t}\left(\left[\Sigma^{1}_{\ell}(\sigma)\right]^{\frac{p-2}{p-1}}+\left[\Sigma_{\ell}^{2}(\sigma)\right]^{\frac{p-2}{p-1}}\right)\left[\Sigma^{3}_{\ell}(\sigma)\right]^{\frac{1}{p-1}}\,d\sigma.
		\end{multline*}
		Applying Hölder’s inequality to the right-hand side with indices $p-1$ and $\frac{p-1}{p-2}$, we get
		\begin{eqnarray}\label{GZA}
			&&	\frac{1}{2}\|\left[u_{\ell}-u_{\infty}\right](t)\sqrt{\rho_{\ell}^{p}}\|^{2}_{2,\Omega_{\ell}}	+\int_{0}^{t}\int_{\mathbb{R}^{N}\times \mathbb{R}^{N}}|v_{\ell}(x,y,\sigma)-v_{\infty}(x,y,\sigma)|^{p}\rho^{p}_{\ell}(x_{1})\,d\mu(x,y)d\sigma\\
			&&	\precsim
			\left[	\left(\int_{0}^{t}\Sigma^{1}_{\ell}(\sigma)\,d\sigma\right)^{\frac{p-2}{p-1}}+\left(\int_{0}^{t}\Sigma^{2}_{\ell}(\sigma)\,d\sigma\right)^{\frac{p-2}{p-1}}\right]\left(\int_{0}^{t}\Sigma^{3}_{\ell}(\sigma)\,d\sigma\right)^{\frac{1}{p-1}}.\nonumber
		\end{eqnarray}
		Since $0\leq \rho_{\ell}\leq 1$, by (\ref{HBV2}) and Lemma \ref{LKM1}, we have
		\begin{equation}\label{QQ1}
			\int_{0}^{t}\Sigma^{1}_{\ell}(\sigma)\,d\sigma=\int_{0}^{t}\int_{\mathbb{R}^{N}\times\mathbb{R}^{N} }|v_{\ell}(x,y,\sigma)|^{p}\rho^{p}_{\ell}(x_{1})\,d\mu(x,y)\,d\sigma\precsim\ell \left\{\|f\|_{L^{2}(0,T;L^{2}(\omega))}^{p'}+\|u_{0}\|^{2}\right\}
		\end{equation}
		and
		\begin{multline*} \int_{0}^{t}\Sigma^{2}_{\ell}(\sigma)\,d\sigma=\int_{0}^{t}\int_{\mathbb{R}^{N}\times\mathbb{R}^{N} }|v_{\infty}(x,y,\sigma)|^{p}\rho^{p}_{\ell}(x_{1})\,d\mu(x,y)\,d\sigma\\\precsim \int_{0}^{t}\int_{\mathbb{R}^{N}\times \mathbb{R}^{N}}\frac{|u_{\infty}(Y_{2},\sigma)|^{p}|\rho_{\ell}(x_{1})-\rho_{\ell}(y_{1})|^{p}}{|x-y|^{N+sp}}\,dxdyd\sigma +\ell \|f\|_{L^{2}(0,T; L^{2}(\omega))}\|u_{\infty}\|_{L^{p}(0,T;L^{p}(\omega))}+\ell \|u_{0}\|^{2}_{2,\omega}.	
		\end{multline*}
		From the estimate of $J^{2}_{\ell}$ in the previous section, we derive
		\begin{equation*}
			\int_{0}^{t}\int_{\mathbb{R}^{N}\times \mathbb{R}^{N}}\frac{|u_{\infty}(Y_{2},\sigma)|^{p}|\rho_{\ell}(x_{1})-\rho_{\ell}(y_{1})|^{p}}{|x-y|^{N+sp}}\,dxdyd\sigma\precsim \frac{\|u_{\infty}\|^{p}_{L^{p}(0,T;L^{p}(\omega))}}{\ell^{sp-1}}.
		\end{equation*}
		Combining this with the above gives
		\begin{equation}\label{QQ2} \int_{0}^{t}\Sigma^{2}_{\ell}(\sigma)\,d\sigma\precsim \frac{\|u_{\infty}\|^{p}_{L^{p}(0,T;L^{p}(\omega))}}{\ell^{sp-1}} +\ell \|f\|_{L^{2}(0,T; L^{2}(\omega))}\|u_{\infty}\|_{L^{p}(0,T;L^{p}(\omega))}+\ell \|u_{0}\|^{2}_{2,\omega}.	
		\end{equation}
		Moreover, note that
		\begin{multline}\label{GVZA}
			\int_{0}^{t}\Sigma^{3}_{\ell}(\sigma)\,d\sigma=\int_{0}^{t}\int_{\mathbb{R}^{N}\times \mathbb{R}^{N}}|u_{\ell}(y,\sigma)-u_{\infty}(Y_{2},\sigma)|^{p}||\rho_{\ell}(x_{1})-\rho_{\ell}(y_{1})|^{p}\,d\mu(x,y)d\sigma\\
			\precsim \int_{0}^{t}J^{1}_{\ell}(\sigma)\,d\sigma +\int_{0}^{t}J^{2}_{\ell}(\sigma)\,d\sigma,\hskip 7cm
		\end{multline}
		where
		$$ J^{1}_{\ell}(\sigma)=\int_{\mathbb{R}^{N}\times \mathbb{R}^{N}}\frac{|u_{\ell}(Y_{2},\sigma)|^{p}|\rho_{\ell}(x_{1})-\rho_{\ell}(y_{1})|^{p}}{|x-y|^{N+sp}}\,dxdy, $$
		and $$ J^{2}_{\ell}(\sigma)=\int_{\mathbb{R}^{N}\times \mathbb{R}^{N}}\frac{|u_{\infty}(Y_{2},\sigma)|^{p}|\rho_{\ell}(x_{1})-\rho_{\ell}(y_{1})|^{p}}{|x-y|^{N+sp}}\,dxdy.$$
		By Lemma \ref{LEm1} and the estimate of  $J^{1}_{\ell} $ from the previous section, we infer
		$$ \int_{0}^{t}J^{1}_{\ell}(\sigma)\,d\sigma\precsim \frac{\|u_{\ell}\|^{p}_{L^{p}(0,T;L^{p}(\Omega_{\ell}))}}{\ell^{sp}}\precsim\frac{1}{\ell^{ps-1}}\left\{\|f\|_{L^{2}(0,T;L^{2}(\omega))}^{p'}+\|u_{0}\|^{2}\right\}.$$
		Similarly, from the estimate of 
		$J^{2}_{\ell}$ we get
		$$ \int_{0}^{t}J^{2}_{\ell}(\sigma)\,d\sigma\precsim \frac{\|u_{\infty}\|^{p}_{L^{p}(0,T;L^{p}(\omega))}}{\ell^{sp-1}}.$$
		Plugging these into \eqref{GVZA} yields
		\begin{equation}\label{GFX}
			\int_{0}^{t}\Sigma^{3}_{\ell}(\sigma)\,d\sigma\precsim \frac{\|u_{\infty}\|^{p}_{L^{p}(0,T;L^{p}(\omega))}+\|f\|_{L^{2}(0,T;L^{2}(\omega))}^{p'}+\|u_{0}\|^{2}}{\ell^{sp-1}}.	
		\end{equation}
		Substituting \eqref{QQ1}, \eqref{QQ2}, and the above into \eqref{GZA} gives
		$$\begin{array}{ll}
			\displaystyle \frac{1}{2}\|\left[u_{\ell}-u_{\infty}\right](t)\sqrt{\rho_{\ell}^{p}}\|^{2}_{2,\Omega_{\ell}}	&\displaystyle +\int_{0}^{t}\int_{\mathbb{R}^{N}\times \mathbb{R}^{N}}|v_{\ell}(x,y,\sigma)-v_{\infty}(x,y,\sigma)|^{p}\rho^{p}_{\ell}(x_{1})\,d\mu(x,y)d\sigma\\
			&\displaystyle \precsim \frac{1}{\ell^{sp-1}}+\frac{1}{\ell^{\frac{sp}{p-1}-1}}.\end{array}$$
		Further, applying the fractional $p$-Poincar\'e inequality to $(u_{\ell}-u_{\infty})\rho_{\ell}$ on the left hand side of the last inequality and arguing as before, we find 
		\begin{multline}\label{GVL}
			\|\left[u_{\ell}-u_{\infty}\right](t)\sqrt{\rho_{\ell}^{p}}\|^{2}_{2,\Omega_{\ell}}	+\|\left[u_{\ell}-u_{\infty}\right]\rho_{\ell} \|^{p}_{L^{p}(0,T; L^{p}(\Omega_{\ell}))}\precsim \frac{1}{\ell^{sp-1}}+\frac{1}{\ell^{\frac{sp}{p-1}-1}}+\int_{0}^{t}\Sigma^{3}_{\ell}(\sigma)\,d\sigma\\
			\stackrel{\text{by \;(\ref{GFX})}}{\precsim}\frac{1}{\ell^{sp-1}}+\frac{1}{\ell^{\frac{sp}{p-1}-1}}.\hskip 2.8cm
		\end{multline}
		Since $\rho_{\ell}=1$ on $\Omega_{\ell/2}$, this implies 
		\begin{equation*}
			\|\left[u_{\ell}-u_{\infty}\right](t)\|^{2}_{2,\Omega_{\ell/2}}	+\|u_{\ell}-u_{\infty}\|^{p}_{L^{p}(0,T; L^{p}(\Omega_{\ell/2}))}\precsim\frac{1}{\ell^{sp-1}}+\frac{1}{\ell^{\frac{sp}{p-1}-1}}.
		\end{equation*}
		Choosing $\ell$  large enough so that  $\ell/2 > \ell_{0}$, we conclude
		\begin{equation*}
			\|u_{\ell}-u_{\infty}\|^{2}_{L^{\infty}(0,T; L^{2}(\Omega_{\ell_{0}}))}	+\|u_{\ell}-u_{\infty}\|^{p}_{L^{p}(0,T; L^{p}(\Omega_{\ell_{0}}))}\precsim\frac{1}{\ell^{sp-1}}+\frac{1}{\ell^{\frac{sp}{p-1}-1}}.
		\end{equation*}
		which completes the proof.
	\end{proof}

	\vspace{2cm}
	
\textbf{Author contributions}  All authors contributed equally to this manuscript.\\
\vspace{0.5cm}\\
\textbf{Availability of data and material} No datasets were generated or analysed during the current study.
\vspace{0.5cm}\\
\textbf{Declarations}\\
\textbf{Competing interests } The authors declare no competing interests.

\end{document}